\documentclass{amsart}
\usepackage{graphicx}

\usepackage{amsmath,amsfonts,amsthm,amssymb,graphics,color,datetime}
\usepackage{hyperref}

\newtheorem{theorem}{Theorem}
\newtheorem{prop}{Proposition}
\newtheorem{lemma}{Lemma}
\newtheorem{cor}{Corollary}

\newtheorem{remark}{Remark}
\newtheorem{defn}{Definition}

\numberwithin{equation}{section}

\newcommand{\pa}{\partial} 
        
        \definecolor{pink}{rgb}{1,0,1}
        \definecolor{purple}{rgb}{0.4,0.2,1}

\newcommand{\Ker}{\operatorname{Ker}}

\newcommand{\LL}{\mathbb{L}}
\newcommand{\cher}{\operatorname{Chern}}

\newcommand{\eps}{\varepsilon}

\newcommand{\cX}{\mathcal{X}}

\newcommand{\tr}{{\rm Tr}}

\newcommand{\cK}{{\mathcal{K}}}

\newcommand{\cT}{{\mathcal{T}}}

\newcommand{\N}{\mathbb{N}}
\newcommand{\R}{\mathbb{R}}

\newcommand{\C}{\mathbb{C}}
\newcommand{\Z}{\mathbb{Z}}

\newcommand{\cL}{\mathcal{L}}

\def\l{\left}

\def\({\left(}
\def\){\right)}
\def\[{\left[}
\def\]{\right]}
\def\la{\langle}
\def\ra{\rangle}

\newcommand{\prim}{\textnormal{Prim}(\Gamma)}
\newcommand{\HH}{\mathbb{H}}
\definecolor{lapislazuli}{rgb}{0.15, 0.38, 0.61}

\begin{document}

\title[Variations of zeta functions and curvature asymptotics]{Second variation of Selberg zeta functions and curvature asymptotics}

\author[K.\@ Fedosova]{Ksenia Fedosova}
\address{Albert-Ludwigs-Universit\"at Freiburg, Mathematisches Institut, Ernst-Zermelo-Str.~1, 79104 Freiburg im Breisgau, Germany}
\email{ksenia.fedosova@math.uni-freiburg.de}

\author[J. \@ Rowlett]{Julie Rowlett}
\address{Mathematical Sciences, Chalmers University of Technology, 412 96 Gothenburg, Sweden}
\email{julie.rowlett@chalmers.se}

\author[G.\@ Zhang]{Genkai Zhang}
\address{Mathematical Sciences, Chalmers University of Technology, 412 96 Gothenburg, Sweden}
\email{genkai@chalmers.se}



\begin{abstract} We give an explicit formula for the second variation of the logarithm of the Selberg zeta function, $Z(s)$, on Teichm\"uller space.  We then use this formula to determine the asymptotic behavior as $\Re s \to \infty$ of the second variation.
  As a consequence, for $m \in \N$, 
  we obtain the complete expansion in $m$ of the curvature of the
    vector bundle $H^0(X_t, \mathcal K_t)\to t\in \mathcal T$ of holomorphic m-differentials over
    the Teichm\"uller space $\mathcal T$, for $m$ large.  Moreover, we
    show that this curvature 
    agrees with the Quillen curvature up to a term of exponential decay, $O(m^2 e^{-l_0 m}),$ where $l_0$ is the length of the shortest closed hyperbolic geodesic.  

\end{abstract}

\maketitle

\section{Introduction}
Selberg was one of many mathematicians for whom investigating the Riemann hypothesis would lead to deep results of broad interest, not only in analytic number theory but also in many other neighboring fields.  To wit, Selberg's trace formula was one of the main inspirations of the Langlands program.  The Selberg zeta function is shrouded with a certain mystique because it is defined in terms of quantities which are in general incomputable, namely the set of lengths of closed geodesics on a Riemannian manifold, 
\begin{equation}\label{SelbergZeta}	Z(s) := \prod_{\gamma \in \prim} Z_\gamma (s), \quad Z_\gamma = \prod_{k=0}^\infty \(1 - e^{-l(\gamma) \cdot (s+k)}\).  \end{equation}
Here, the geometric setting is a compact Riemann surface, $X$, of genus $g\ge 2$, equipped with the hyperbolic metric of curvature $-1$. Let $\Gamma$ be the fundamental group of $X$.  We may then fix $X$ as the quotient of the upper half plane,
$\HH=\{z=x+iy, y>0\}$, by $\Gamma$,  so that $X=\Gamma\backslash \HH$.  We say a hyperbolic element, $\gamma \in \Gamma$, is primitive if for all $\gamma_0 \in \Gamma$ and $k \in \N$ with $\gamma = \gamma_0^k$ it follows that $\gamma_0 = \gamma$ and $k=1$.  Then $\prim $ in (\ref{SelbergZeta}) is the set of conjugacy classes of primitive hyperbolic elements $\gamma$ of $\Gamma$, which is in canonical bijection with the set of primitive closed geodesics, and $\ell(\gamma)$ is the geodesic length of the associated conjugacy class of $\gamma$.  

It is clear from \eqref{SelbergZeta} that the Selberg zeta function is intimately linked to the Riemannian geometry of $X$.  What is perhaps not so obvious is that it is also closely connected to the complex structure of $X$.  To describe this, we fix $S$, the so-called \em model surface \em of genus $g \geq 2$.  The Teichm\"uller space $\cT = \cT_g$ of surfaces of genus $g$ is the set of equivalence classes $[ (\Sigma, \varphi)]$, where $\Sigma$ is a Riemann surface, and $\varphi : S \to \Sigma$ is a diffeomorphism, known as a \em marking.  \em  On each such surface, $\Sigma$, there is a unique Riemannian metric which has constant curvature $-1$, however in this notation the Riemannian metric is suppressed.  The equivalence relation identifies 
$$(\Sigma_1, \varphi_1) \sim (\Sigma_2, \varphi_2)$$
if there is an isometry $I: \Sigma_1 \to \Sigma_2$ such that $I$ and $\varphi_2 \circ \varphi_1^{-1}$ are isotopic.  Hence, from the Riemannian geometric perspective, these two surfaces are identical, in that they are topologically the same, and they are equipped with the same Riemannian metric.

One may also consider $\cT$ from a complex analytic perspective.  For this purpose, we recall that a Beltrami differential, $\mu$, is a $\Gamma$ invariant $\pa_z \otimes d \bar z$ tensor on $\HH$, thus we write $\mu = \mu(z) \pa_z \otimes d \bar z$. 
 It is harmonic if $\mu(z)=\overline{\phi(z)} y^2$, and $\phi=\phi(z) dz^2$
 is a $\Gamma$ invariant holomorphic quadratic differential.  Ahlfors 
\cite{Ahlfors-1961} showed that tangent vectors in $\cT_t$ for a point $t = [(X, \varphi)] \in \cT$ are represented by harmonic Beltrami differentials.  To see this, for a harmonic Beltrami differential $\mu$, let $f^\mu$ be the solution of the Beltrami equation 
\begin{equation} \label{beltramieq}
\begin{split}  
f_{\bar z}  &=  \mu(z)f_z, \quad z \in \HH \\
f_{\bar{z}}  &=  \overline{\mu(\bar{z})} f_z, \quad z \in \LL \\
f(0)  &=  0, \quad f(1)=1, \quad f(\infty) = \infty,
\end{split}
\end{equation}  
where $\LL$ is the lower half plane in $\C$.  For a fixed Beltrami differential $\mu$ of (supremum) norm $1$, let $\eps$ be a small complex number.  Consider the Beltrami equation for $\eps \mu$ with solution $f^{\eps \mu}$.  Then, for sufficiently small $\eps$, $f^{\eps \mu}$ defines a Fuchsian group $\Gamma^\eps = f^{\eps \mu} \Gamma (f^{\eps \mu})^{-1}$.  The Riemann surfaces $X^\eps = \Gamma^\eps \backslash \HH$ define a curve in $\cT$ with $X^0 = X$.  Hence, we identify unit tangent vectors in $\cT_t$ with harmonic Beltrami differentials of unit norm.  Each of these in turn defines a local one parameter family of Riemann surfaces, $X^\eps$.  In this way we compute the variation of quantities defined on the surface $X$ corresponding to the point $t \in \cT$ in the directions corresponding to these harmonic Beltrami differentials of unit norm.  The stage is now set to present our main results.  

\subsection{Main results} 
Our first main result generalizes \cite[Theorem 1.1.2]{Gon} in which Gon computed a formula for the first variation of the log of the Selberg zeta function; this may be compared with our Proposition \ref{firstvarlogszf}.  The variation is computed, in both our setting and that of Gon, by viewing the Selberg zeta function as a function on Teichm\"uller space, $\cT$, defined for $t \in \cT$ as the Selberg zeta function on the corresponding Riemann surface equipped with the hyperbolic Riemannian metric of constant curvature $-1$.  

\begin{theorem}\label{second_variation_geod_sum}
For $\Re(s)>1$, we have   
\begin{equation} \bar\partial_\mu\partial_\mu \log Z(s)= \sum_{\gamma \in \prim} \bar\partial_\mu\partial_\mu \log \ell(\gamma) A_\gamma(s)+ \sum_{\gamma \in \prim} |\partial_\mu \log \ell(\gamma)|^2(A_\gamma(s) + B_\gamma(s) )
  \end{equation}
where
$$A_\gamma(s)=s \frac{d}{ds} \log Z_\gamma (s) +\frac{d}{ds} \log z_\gamma (s)^{-1}=\sum_{k=0}^\infty\frac{(s+k) \ell(\gamma)}{e^{(s+k)\ell(\gamma)}-1},$$
and
$$B_\gamma(s)= \left( s^2\frac{d^2}{ds^2}\log Z_\gamma(s)+2s\frac{d^2}{ds^2}\log  z_\gamma(s)^{-1}
+\frac{d^2}{ds^2}\log  \tilde z_\gamma(s)^{-1}\right) $$
$$= - l(\gamma)^2\sum_{k=0}^\infty\frac{(s+k)^2 e^{(s+k)\ell(\gamma)}}{(e^{\ell(\gamma)(s+k)}-1)^2}.$$ 
Above $z_\gamma(s)$ and $\tilde z_\gamma(s)$ are as in (\ref{localzeta}) and (\ref{localhighzeta}), respectively.
\end{theorem}
To avoid cumbersome notation, we have not included the explicit formulas for the first and second variations of the lengths of closed geodesics in the second variational formula above.  These are contained in \S \ref{s:vargeo}, Propositions \ref{lengthvarprop1} and \ref{lengthvarprop2}, respectively.  

In our next main result, we prove asymptotics of the second variation of $\log Z(s)$ for $\Re(s) \to \infty$.  To state the result, we require the set of systole geodesics, 
\begin{equation} \label{systoles} S(X) := \{ \gamma \in \prim : l(\gamma) = l_0 \},  \quad l_0 = \inf \{ l(\gamma) : \gamma \in \prim \}. \end{equation} 
Correspondingly, listing all the systole geodesics as
 $S(X)=\{l(\gamma_1), \cdots, l(\gamma_N)\}$, we define 
\begin{equation}
\label{part-l_0}
\partial l_0: \mu\in T_t^{(1, 0)}(\cT)\mapsto
(\partial_{\mu} l(\gamma_1), \cdots, \partial l_{\mu}(\gamma_N)) \in
\mathbb C^N. 
\end{equation}
and 
\begin{equation} \label{abusenot} |\pa_\mu l_0|^2 := \sum_{S(X)} |\pa_\mu l(\gamma)|^2, \end{equation}
and 
\begin{equation} \label{abusenot2} 
\bar \partial_\mu \partial_\mu \log l_0 := \sum_{S(X)} \bar \partial_\mu \partial_\mu \log l(\gamma).
\end{equation}

\begin{theorem}\label{thmasy} 
If $|\pa_\mu l_0|^2 \neq 0$, then 
\begin{equation} \label{asy1}  \lim_{\Re s \to \infty} \frac{\bar\partial_\mu\partial_\mu \log Z(s)} {  s^2 e^{-s l_0}}=-\frac{|\pa_\mu  l_0|^2}{1-e^{-l_0} } < 0. \end{equation} 
For the Ruelle zeta function, $R(s)$, which is defined in \eqref{ruellezeta}, 
$$\lim_{\Re s \to \infty} \frac{\bar\partial_\mu\partial_\mu \log R(s)} {  s^2 e^{-s l_0}}=-\frac{|\pa_\mu  l_0|^2}{1-e^{-l_0} } < 0.$$
The square of the Hilbert-Schmidt norm of the resolvent, which is defined in Lemma \ref {le:hsnorm}, 
satisfies 
$$\lim_{\Re s \to \infty} \frac{\bar \pa_\mu \pa_\mu ||(\Delta_0 + s(s-1))^{-1}||_{HS} ^2}{e^{-s l_0}} = \frac{ |\pa_\mu l_0|^2 l_0^2 }{4(1-e^{-l_0})} > 0.$$

In case $|\pa_\mu l_0|^2 = 0$, then we have 
\begin{equation} \label{asy1} 
\lim_{\Re s \to \infty} \frac{ \bar\partial_\mu\partial_\mu \log Z(s) }{s  e^{-s l_0}} = 
\frac{l_0 }{1-e^{-l_0}}\bar \pa_\mu \pa_\mu \log l_0 >0. \end{equation} 
In this case, we also have for the Ruelle zeta function, 
$$
\lim_{\Re s \to \infty} \frac{ \bar\partial_\mu\partial_\mu \log R(s) }{s  e^{-s l_0}} = 
\frac{l_0 }{1-e^{-l_0}}\bar \pa_\mu \pa_\mu \log l_0 >0.
$$
The square of the Hilbert-Schmidt norm of the resolvent in this case satisfies 
$$\lim_{\Re s \to \infty} \frac{\bar \pa_\mu \pa_\mu ||(\Delta_0 + s(s-1))^{-1}||_{HS} ^2}{ \frac 1 s e^{-s l_0}} = -  \frac{ l_0^2 \bar \pa_\mu \pa_\mu \log l_0}{4 (1-e^{- l_0})} < 0.$$

The zeta regularized determinant, $\det (\Delta_0 + s(s-1))$ satisfies 
$$\bar \pa_\mu \pa_\mu \log \det(\Delta_0 + s(s-1)) = \bar\partial_\mu\partial_\mu \log Z(s),$$
and therefore analogous results hold for its behavior as $\Re s \to \infty$. 
\end{theorem}

It follows from the above theorem
  that the  Hessian
  $  \bar \partial \partial \log Z(s)$
  for large $s\in \mathbb R$ is
  not positive definite on Teichm\"uller space, $\mathcal T$.
  More precisely for each fixed $t\in \mathcal T$ let $r=r(t)$ be the
  rank of the linear map (\ref{part-l_0}).
  It follows from the result of Wolpert
  \cite[Theorem 8]{Wolpert-JDG} that the differential
  $dl$ of the  length
  function  of one geodesic $l$ is everywhere non-vanishing.
  In particular the $\mathbb R^N$-valued differential $\partial l_0$
  defined in (\ref{part-l_0})
 is non-vanishing, 
  and $1\le r\le 3g-3$. It is also known (see \cite[Theorem 3]{APP},
  \cite{SchmutzSchaller-JDG})
  that there exists
  $t\in \mathcal T$ such that the number $N$ of systole geodesics
is bounded above by $2g$, $N\le 2g$, thus the rank $r\le 2g$ at $t$.
Our result states then that if $\mu\in T_t(\mathcal T)$ is
in  the subspace $\text{Ker} \partial l_0$ then
the Hessian $\bar   \partial_{\mu}\partial_{ \mu} \log Z(s)$
has positive sign for large $s \in \R$, and
if $\mu$ is
in the orthogonal complement $(\text{Ker} \partial l_0)^{\perp}$ then
Hessian has negative sign for large $s \in \R$, the dimension of the two spaces
being $r$ and $3g-3-r$, respectively. 

Assume now that $m \in \N$.  As a consequence of Theorem
\ref{thmasy}, we prove that the curvature, $\cher^{(m)}(\mu, \mu)$, of the vector bundle $H^0(X_t, \mathcal K_t)\to t\in \mathcal T$
  of holomorphic m-differentials over the Teichm\"uller space, $\cT$, agrees with the Quillen curvature up to a term of exponential decay.  In particular we obtain the full expansion of the $\cher^{(m)}(\mu, \mu)$ in $m$.  

\begin{cor} \label{cor2}
The  curvature, $\cher^{(m)}(\mu, \mu)$, of the vector bundle
$H^0(X_t, \mathcal K_t)\to t\in \mathcal T$
 over the Teichm\"uller space, $\cT$, has the following expansion, 
\begin{equation} \label{eq:thm-3-2}
\cher^{(m)}(\mu, \mu)=\frac{6m(m-1)+1}{12\pi}||\mu||_{WP} ^2 + R(m), \quad m \to \infty.  
\end{equation}
Here, $\Vert \mu \Vert_{WP} ^2$ is the square of the Petersson norm of $\mu$.  The remainder 
$$R(m) = O(m^2 e^{-m l_0}), \quad m \to \infty.$$
\end{cor}

The corollary shows that  $\cher^{(m)}(\mu, \mu)$ and the Quillen curvature agree up to an exponentially small remainder term.  This improves, in the case of Riemann surfaces, the result of Ma-Zhang \cite{Ma-Zh} where the first two terms were found.  In this sense, our result can be seen as a variational version of the result of Bismut-Vasserot \cite{bismut1989} on the asymptotics of analytic torsion.  It is also closely related to the curvature of the Quillen metric on Teichm\"u{}ller space which has been studied in the general context of holomorphic families of K\"a{}hler manifolds \cite{bisbost}.
	
\subsection{Related works}  
To the best of our knowledge, the first result on plurisubharmonicity of naturally defined functions on Teichm\"uller space appeared in \cite{wolpert1987}.  There, Wolpert showed that geodesic length functions are plurisubharmonic as functions on Teichm\"uller space. 
He has built upon and generalized those results in \cite{wolpert2006, wolpert2008}.  More recently, Axelsson \& Schumacher \cite{Ax-Sch-2012}, using K\"ahler geometric methods, established the plurisubharmonicity of each geodesic length function.  They obtained this result as a corollary to formulas they demonstrated for the first and the second variations of the geodesic length as a function on Teichm\"uller space.   The Weil-Petersson Hession of length was also studied by Wolf \cite{Wolf-JDG-2012} at the same time as Axelsson \& Schumacher obtained their results.

Closely related to our work is that of Gon \cite{Gon}.  That context is more general because the underlying model surface in the definition of Teichm\"uller space is of type $(g,n)$, that is genus $g$ and punctured at $n$ points.  Our work is only for $(g,0)$, so that our surfaces are not punctured.  The main result of \cite{Gon} expresses the variation $\partial_\mu \log Z_\Gamma (s)$ as a sum over conjugacy classes of primitive hyperbolic elements of certain quantities depending on local higher zeta functions and periods of the automorphic forms over the closed geodesic.  The formula was obtained with the help of Takhtajan and Zograf's integral expression for the first variation of the Selberg zeta formula, 
\begin{equation}\label{first_var_integral}
\partial_\mu \log Z_\Gamma (s) = \int_X F_s(z) \mu.  
\end{equation}

Above, $F_s$ is a certain Poincar\'e series constructed from a second derivative of the resolvent kernel $Q_s(z,z')$ of the Laplacian on the surface. More precisely, 
$$
F_s = \sum_{\gamma \neq e, \gamma \in \Gamma} \left. \frac{\partial^2}{\partial z \partial z'} Q_s(z, \gamma z') \right|_{z = z'}, 
$$
where $e$ denotes the identity element.  Using the explicit formulas for $Q_s(z, \gamma z')$, Gon was able to express $\partial_\mu \log Z_\Gamma(s)$ via the sum of local Selberg zeta functions \cite[Theorem 1.1.2]{Gon}. 

Our method is different. Instead of exploiting the integral formula of Zograf and Takhtajan, we show that it is possible to differentiate the definition of the Selberg zeta function, directly, as long as it converges.  This serves our purposes well, because we are interested in arguments for which the product (\ref{SelbergZeta}) converges.  Next, we use the result of Axelsson and Schumacher \cite{Ax-Sch-2012} for the variation of the length of a geodesic.  Interestingly, although our method is different, we arrive at the same formula as Gon.  Moreover, our approach gives a geometric interpretation for the automorphic forms appearing in Gon's formula:  the automorphic forms correspond to the variations of the lengths of the geodesics.  Furthermore, with our method we are able to calculate the second variation of the Selberg zeta function as a sum over primitive closed geodesics of the surface.

Consequently, this method is applicable not only for studying the  Selberg zeta function, but also to study all other functions that are defined in an analogous way.  In particular, our techniques apply equally well to functions which are defined as a sum or product, over primitive closed geodesics, of quantities depending on the lengths of closed geodesics.  To illustrate the utility of our method, we compute in \S \ref{randomvars} variational results for:  the trace of the squared resolvent, the Ruelle zeta function, the zeta-regularized determinant of the Laplacian, and the hierarchy of higher Selberg zeta functions.  For the definition of these higher zeta functions, see (\ref{hierSelbergZeta}).  These form a real-parameter family $Z(s, t)$ of zeta functions which generalize the notion of the Selberg zeta function.  Interestingly, the local version of these higher Selberg zeta functions also appear in Gon's formula.  Here, we prove directly using their definitions a variational formula for the whole hierarchy which relates the variation of $Z(s,t)$ to that of $Z(s,t')$ for $t'<t$.  

We would also like to mention related results obtained by Fay \cite{Fay}, who considered Selberg zeta functions twisted by a representation of $\Gamma$.  It may be possible to generalize our results to that setting as well, if the key elements in the proofs are amenable to suitable adaptations.

\subsection{Key elements in the proofs}
Initially, we prove the first and second variational formulas, Proposition \ref{firstvarlogszf} and Theorem \ref{second_variation_geod_sum}, respectively, by differentiating the definition of the logarithm of the Selberg zeta function and verifying convergence.  As we proceed directly using the sum over closed geodesics, we must compute the first and second variation of the length of each closed geodesic.  To compute the first variation, although this would follow from \cite{Ax-Sch-2012}, we compute in a more classical and direct way using Gardiner's formula, obtaining an equivalent but superficially different formula.  However, unlike the first variational formula of \cite{Ax-Sch-2012}, one can read-off the terms in Gon's formula directly from our Proposition \ref{lengthvarprop1}.  

The asymptotics of the second variation of $\log Z(s)$ are obtained by locating the dominant term in our formula as $\Re s \to \infty$. When $s=m\ge 0$ is an integer, the variation is the difference of the curvature of the vector bundle $H^0(\Sigma_t, \mathcal K_t^m)\to t\in \mathcal T
$
over the Teichm\"uller space $\mathcal T$
and the Quillen curvature. The Quillen curvature is well-known and is given by a second degree polynomial in $m$.  We therefore obtain the full expansion of the curvature $\cher^{(m)}(\mu, \mu)$.  Apart from the case of abelian varieties this seems the first case  where a full expansion of $\cher^{(m)}(\mu, \mu)$ has been obtained.


\subsection{Further developments} 
In a subsequent paper \cite{FRWZ} we shall demonstrate an integral formula for the second variation which holds for $\Re(s)>1$, in the spirit of the integral formula of Takhtajan \& Zograf, \cite[Theorem 2]{ZT1987} for $s = m \in \N$.   We shall use this formula to define the  curvature,  $\cher^{(m)}(\mu, \mu)$, for non-integer $m$.  The Teichm\"uller space
andthe vector bundle
  $H^0(\mathcal K^m_t)
  \mapsto t\in \mathcal T$
  of
  holomorphic m-differentials   $H^0(\mathcal K^m_t)$
  over the Teichm\"uller space 
can be formulated in the general setup
of relative ample line bundles for fibrations of K\"a{}hler manifolds \cite{Bob}.  In a recent preprint \cite{Wan-Zhang} the third auhor together with Wan has  been able to prove a generalization of Corollary \ref{cor2} in this general setup. 

\subsection{Organization} 
In the next section we gather definitions and notations and demonstrate the requisite preliminaries.  In \S 3 we prove the first variational formula as well as estimates and a convergence result which will be used to justify convergence of the second variational formulas demonstrated in \S 4. The asymptotics of the second variation for  $\Re s \to \infty$ are computed in \S 5, which are then used to compute the asymptotics of the curvature of a Hermitian holomorphic vector bundle.  We conclude with an investigation of the special cases $s=m \in\{1, 2\}$.  Finally, in the appendix we provide a calculation of the Hilbert-Schmidt norm of the squared resolvent.  Although the formula is known, our particular method of calculation is not contained in the literature to the best of our knowledge and therefore may be interesting or useful.  This calculation is used to compute the variation of the Hilbert-Schmidt norm in \S 4.  

\subsection{Acknowledgements} We are very grateful to Bo Berndtsson
for several inspiring discussions and to Steve Zelditch for clarifying
some concepts in  Teichm\"u{}ller theory.  We thank also Werner
M\"u{}ller  for drawing our attention to the reference \cite{Fay}.  We
appreciate stimulating discussions with Dennis Eriksson and  a careful
reading of the paper by  Magnus Goffeng and  Xueyan Wan.  The second
author is supported by the Swedish Research Council Grant 2018-03873, the third author by the Swedish
Research  Council Grant 2018-03402.  The second author gratefully acknowledges the support of the National Science Foundation Grant DMS-1440140 as well as a room with a view at the Mathematical Sciences Research Institute in Berkeley, California during the fall 2019 semester.  All authors are grateful to the comments of the anonymous reviewers which have resulted in significant improvements to the quality of the paper.   

\section{Preliminaries}
We  fix notations and prepare the technical tools required for our proofs. 

\subsection{Hodge and $\bar\partial$ Laplacians on $L^2_{m, l}(X)$.}
We briefly recall a few known results for certain Laplace operators on Riemann surfaces which shall be important ingredients in the proofs of our results. Each point in Teichm\"uller space $ \cT = \cT_g $ is canonically identified with a compact Riemann surface, $X$ of genus $g \geq 2$, which admits a unique Riemannian metric of constant curvature $-1$.  Then, $X$ is identified with the quotient $\Gamma \backslash \HH$, where $\Gamma$ is the fundamental group of $X$, and $\HH$ is the upper half plane in $\C$.  The hyperbolic metric in Euclidean coordinates on the upper half plane is given by 
\begin{equation} \label{hypmetric} \rho(z)|dz|^2, \quad \rho(z)=y^{-2}, \quad  \HH = \{ z = x+iy \in \C : \Im z = y > 0\}. \end{equation} 
Let $\Delta_0=-y^2(\partial_x^2 +\partial_y^2)$
be the Laplace operator on scalar functions.  We note that 
$$\Delta_0 = -4 y^2 \frac{\pa^2}{\pa z \pa \bar z}.$$ 
For this reason, there are different normalizations of the Laplacian by different authors.  Particularly relevant to our work is the definition in \cite{TZ-cmp}, who defined the Laplacian as 
$$-y^2 \frac{\pa^2}{\pa z \pa \bar z} = \frac{1}{4} \Delta_0.$$

Let $\mathcal K$ be the holomorphic cotangent bundle and
$\bar {\mathcal K}$ the anti-holomorphic cotangent bundle over~$X$.  We also use the standard notation that $\mathcal K^{-k} = (\mathcal K^k)^*$ is the dual bundle of $\mathcal K^k$.  The scalar product on sections  of $\mathcal K^k \otimes \bar{\mathcal K}^l$  is given by
\begin{equation} \label{integration}  \langle f,g\rangle_{k,l} := \int_X  \langle f, g \rangle_z \rho(z) (\frac i2) dz\wedge d\bar z =  \int_X f \bar{g} \rho^{1-k-l} dA.  \end{equation} 
The integration above is with respect to the Euclidean measure, $dA = dx dy$, and is taken over a fundamental domain of $X$.  In the case of functions, we note that $k=l=0$, and we may simply write the integral of a function $\varphi$ on $X$ as $\int_X \varphi$, suppressing the $\rho dx dy$. 

Let  $L^2_{k, l}(X)$  be the corresponding Hilbert space. We write
$L^2_{k}(X) =L^2_{k, 0}(X)$. Denote by $\nabla'_k$  the $(1, 0)$ part
of the Chern connection, $\nabla_k=\nabla'_k+
\nabla ''$,  $\nabla''=\bar \partial$.
Let $\square = \square'
=(\nabla')^\ast\nabla' 
+\nabla' (\nabla')^\ast$
 and $\square''={\bar\partial}^\ast\bar\partial$
be the corresponding Laplace operators.
The Chern connection 
$\nabla=\nabla' +\bar\partial$ can be extended
to sections of $\mathcal K^m 
\otimes \bar{\mathcal K}$ 
as $(0, 1)$-forms with values in 
$\mathcal K^m $. 
Note that $\square''
={\bar\partial}^\ast
{\bar\partial} +
{\bar\partial}
{\bar\partial}^\ast= 
{\bar\partial}
{\bar\partial}^\ast
$ on $(0, 1)$-forms.

Let $H^0 (\mathcal K^k) = \Ker \bar \partial$
be the space of holomorphic $k$-forms, $k\ge 0$,
and $H^{(0, 1)}(\mathcal K^{k}) 
$ the $\bar\partial$-cohomology
of $(k, 1)$-forms, $k\le 1$. Elements in
$H^{(0, 1)}(\mathcal K^{-1}) $ are represented by the 
harmonic Beltrami differentials $\mu =\mu(z) (dz)^{-1} d\bar z$, and
are identified with 
$H^0 (\mathcal K^2)$ via the duality, 
\begin{equation} \label{phi-mu} 
\phi(z) := \rho(z) \overline{\mu(z)}
=y^{-2}\overline{\mu(z)}, \quad \phi := \phi(z) (dz)^2 \in H^0 (X_t, \cK^2). \end{equation} 

The following
Lemma
is a consequence of the general well-known
Kodaira-Nakano type formulas \cite[Chapter VII, Section 1]{Dem}.
For completeness we provide an elementary proof.

\begin{lemma} \label{lemma1} 
On the space $L^2 _{k, 1} (X)$  of $(0, 1)$
forms with coefficients in  $\mathcal K^k$
the operators $(\nabla')^*\nabla'$,
$\nabla' (\nabla')^*$, and $\square''={\bar\partial} {\bar\partial}^\ast$ are related by
$$
(\nabla')^*\nabla' 
=
\nabla' (\nabla')^* +\frac{k-1}2 ={\bar\partial}
{\bar\partial}^\ast.
$$
\end{lemma}

\begin{proof} 
The Chern connection $\nabla'$ (also called Maass operator) is
given  by $
\nabla'
\left( f(z) (dz)^k d\bar z\right)
= y^{-2k}\partial (f(z) y^{2k}) (dz)^{k+1} 
d\bar z$ and, 
$\bar\partial \left( f(z) (dz)^k\right)
=\bar\partial f(z) (dz)^k d\bar z$.
Thus their adjoints are
$$
(\nabla')^\ast \left( f(z) (dz)^{k+1} d\bar z\right)
=- \bar\partial( y^2f(z)) (dz)^{k} d\bar z
$$
and 
$$
(\bar \partial)^\ast \left( f(z) (dz)^{k} d\bar z\right)
=- y^{2-2k}\partial( y^{2k}f(z)) (dz)^{k}.
$$
We have thus 
$$
(\nabla')^\ast\nabla'
\left( f(z) (dz)^{k} d\bar z\right)
=\bar\partial{\bar\partial}^\ast\left( f(z) (dz)^{k} d\bar z\right)
=g(z)(dz)^{k} d\bar z, \quad
g(z)=-\bar\partial \left(y^{2-2k}\partial( y^{2k} f(z)) \right),
$$
and
$$
\nabla' (\nabla')^\ast
\left( f(z) (dz)^{k} d\bar z\right)
=h(z)(dz)^{k} d\bar z, \quad h(z)=-y^{-2(k-1)}
\partial
\left(
y^{2(k-1)} \bar\partial (y^2 f(z))\right).
$$
The operators 
 $f=f(z) (dz)^k d\bar z\mapsto g=(\nabla')^\ast\nabla' f$, and $f \mapsto h=\nabla' 
(\nabla')^\ast f$ then differ by a constant.  More precisely
$$
g(z)= -y^2\bar\partial\partial f(z)
-i y\partial f(z) +i k y\bar \partial f(z)
-\frac k 2 f(z),
$$
$$ 
h(z)=-y^2\bar\partial\partial f(z)
-i y\partial f(z) +i k y\bar \partial f(z) -\frac{2k-1} 2 f(z)
$$
and $g(z)= h(z) + \frac {k-1} 2f(z)$.
This completes the proof.
\end{proof}

\subsection{The Weil-Petersson Metric on Teichm\"uller space}  
We introduce the Weil-Petersson metric following \cite{tromba}, and we take the opportunity to recall as so nicely done there the origins of this metric.  Petersson introduced an inner product on the spaces of modular forms of arbitrary weight in the context of number theory.  Observing that modular forms of weight two are precisely holomorphic quadratic differentials, Andr\'e Weil remarked in a letter to Lars Ahlfors that Petersson's inner product should give rise to a Riemannian metric on Teichm\"uller space.  This is indeed the case, and Ahlfors went on to prove \cite{ahlfors-curv} that the holomorphic sectional curvature and the Ricci curvature of $\cT$ with respect to this metric, known as the Weil-Petersson metric, are both negative.  However, this metric is not complete, which was demonstrated by Wolpert \cite{wolp-incom}.  

At each point $t \in \cT$, we may identify the holomorphic tangent
vectors at $t$ with harmonic Beltrami differentials, $\mu$. The 
Weil-Petterson metric is defined by
$$\langle \mu, \mu \rangle =   \langle \mu, \mu \rangle_{WP} := 
\int_X |\mu|^2 =\int_X |\mu(z)|^2 \rho(z) dA(z).
$$
This defines a Hermitian metric on $\cT$, which taking the real part defines a Riemannian metric on $\cT$, but we shall only be interested in the Hermitian Weil-Petersson metric on $\cT$.  


\subsection{The curvature of vector bundles on Teichm\"uller space}  \label{sect-curvature}
For each point $t \in \cT = \cT_g$, the Teichm\"uller space of marked surfaces of genus $g$, we denote the  corresponding Riemann surface as $X_t$.  When it is clear from context, we may simply write $X$.   The  holomorphic tangent vectors at each $t\in \mathcal T$, $\mu \in T^{(1,0)} _t ( \cT)$, are identified with harmonic Beltrami differentials $\mu \in H^{(0, 1)}(X_t, \mathcal K^{-1})$.  We may also identify $\cK^{-1}$ with $\overline \cK$, with the observation that $(\pa_z)^* (\pa_z) = 1 = dz \otimes \pa_z$.  Thus we write in terms of a local coordinate, $z$, 
$$T_t ^{(1,0)} (\cT) \ni \mu=\mu(z) d\bar{z} {\partial_z}.$$

The coholomogy $H^{(0, 1)}(X_t, \mathcal K^{-1})$ is identified
further with 
$H^0 (X_t, \cK^2)$ as in (\ref{phi-mu}).
In this way, we have an anti-complex linear identification between
$T_t ^{(1,0)} (\cT)$
 with the holomorphic quadratic
differential $H^0 (X_t, \cK^2)$, namely the dual of 
$T_t ^{(1,0)} (\cT)$ with $H^0 (X_t, \cK^2)$.

Let $m\ge 1$.   The map $H^0(X_t, \mathcal K_t^m)\mapsto t \in \mathcal T$ can be used to define 
a  holomorphic Hermitian vector bundle over the Teichm\"uller space $\mathcal T$.  More precisely, $p : \cX \to \cT$ is a smooth proper holomorphic fibration of complex manifolds of (complex) dimensions $3g-2$ and $3g-3$, respectively, such that for each point $t \in \cT$, the fiber over $t$ is the surface, $X_t$.  In the language of \cite{Bo-mz}, $Y = \cT$, and the relative (complex) dimension, $n=1$.  We then consider the holomorphic line bundle $\cL \to \cX$ such that $\cL|_{\cX_t} = \cK^{m-1} _{X_t}:=\mathcal K_t$.  It is well known that $\cL \to \cX$ is equipped with a smooth metric of positive curvature; the positivity follows from \cite[Lemma 5.8]{wolpert86}.  Moreover, Wolpert also showed in that work that $\cX$ is equipped with a K\"ahler metric. 

The direct image sheaf of the relative canonical bundle twisted with $\cL$, 
$$p_* (\cL + \cK_{\cX/\cT})$$ 
is then associated to the vector bundle, $E$, over $\cT$, with fibers
$$E_t = H^0 ( \cX_t, \cK_{\cX_t} + \cL|_{\cX_t}) = H^0 (X_t, \cK^m _{X_t}).$$
Thus, an element in $E_t$ is a holomorphic $(1,0)$ form, $u$, on $X_t$ with values in $\cL|_{X_t}$.  The Kodaira-Spencer map at a point $t \in \cT$ is a map from the holomorphic tangent space to the first Dolbeault cohomology group, $H^{0,1} (X_t, T^{1,0} (X_t))$ of $X_t$ with values in the holomorphic tangent space of $X_t$.  The image of this map is known as the \em Kodaira-Spencer class, \em $K_t$.  This class has a natural action on $u \in E_t$, which is denoted by $K_t \cdot u$.  If $k_t$ is a vector-valued $(0,1)$ form in $K_t$, then  
\begin{equation} \label{kodaira} K_t \cdot u := [k_t \cdot u] \in H^{(0,1)} (X_t, \cK^{m-1} _{X_t}). \end{equation} 
In this setting, we note that the harmonic Beltrami differential, $\mu$, is in the Kodaira-Spencer class, $K_t$, so we may take $k_t = \mu$ above.  Then, the action $\mu \cdot u$ is well-defined defined via $k_t \cdot u$.  

We recall the curvature formula of Berndtsson  \cite{Bob, Bo-mz} for general relative direct image bundle specified
to our case above.  Interestingly, it is precisely the curvature of this bundle which shall appear in our second variational formula for the logarithm of the Selberg zeta function at integer points.  Fix $X=X_t$
and denote  
\begin{equation} \label{fmu} 
f({\mu})=(1+\square_0)^{-1}|\mu|^2 \end{equation}  
where $\square_0=2\bar\partial^*\bar\partial$ is the Laplacian 
on functions.  Although $\mu(z)$ is not a well-defined pointwise function, $|\mu|^2 = |\mu(z)|^2 |d\bar z|^2 |\pa_z|^2$ is indeed well-defined pointwise. Similarly for an element $u \in E_t$, $|u|^2$ is also a well-defined pointwise function.  

\begin{prop} \label{Bob-prop}  The curvature $R^{(m)}(\mu, \mu)$ of
  the bundle  $H^0(X_t, \mathcal K_t^m)\to t\in \mathcal T$ is given by 
\begin{equation}  \label{eq:bob-1} \langle  R^{(m)}(\mu, \mu) u, u\rangle =\Vert [ \mu\cdot u ]\Vert^2, \quad \textrm{for } m=1.  
\end{equation}
Above, $[\mu \cdot u]$ is defined in (\ref{kodaira}), and the norm is with respect to the natural K\"ahler metric on the bundle $E$ induced by the K\"ahler metric on $\cX$.  The norm is taken with respect to the unique harmonic representative in the class $[\mu \cdot u]$.  For $m \ge 2$, the curvature 
 \begin{equation}  \label{eq:bob-2} \langle  R^{(m)}(\mu, \mu)u,
   u\rangle =(m-1)\int_{X} f(\mu)|u|^2 +
 \frac{m-1}2 \langle \left(\square''+\frac{m-1}2 \right)^{-1}(\mu\cdot u),  (\mu\cdot u) \rangle.  \end{equation}
Above, $\square''=\bar\partial{\bar\partial}^\ast
=(\nabla')^\ast \nabla'
$ is the $\bar\partial$-Laplacian acting on $(m-1, 1)$-forms
$\mu\cdot u$ in Lemma  \ref{lemma1}.

Letting $\{u_j\}_{j=1} ^{d_m}$ be an orthonormal basis of $H^0(\mathcal K^m)$, the curvature, 
$$ \cher^{(m)} (\mu, \mu)= \tr \, R^{(m)}(\mu, \mu) $$ 
is given  by
$$ \cher^{(m)} (\mu,  \mu) =\sum_{j=1} ^{d_m} ||[\mu\cdot u_j]||^2, \quad \textrm{for } m=1,$$
where the norm is the same as in (\ref{eq:bob-1}).  For $m\geq 2$, the curvature 
$$\cher^{(m)}(\mu, \mu) =I^{(m)}  + II^{(m)}$$
with 
$$I^{(m)}=(m-1) \sum_{j=1} ^{d_m} \int_X f(\mu) |u_j|^2 $$
and 
$$ II^{(m)}=\frac{m-1}2\sum_{j=1} ^{d_m} \langle \left(\square'' +
\frac{m-1}2\right)^{-1}(\mu\cdot u_j), \mu\cdot u_j\rangle $$
for $m\ge 2$. 
\end{prop} 
\begin{proof}  For $m=2$, this curvature formula is due to Wolpert; see \cite[Theorem 4.2]{wolpert86}.   The results are proved in  \cite{Bob, Bo-mz} for the general
setup, and we specify them to our case.  We also show how the result in \cite{Sun} 
for our case is a consequence of the general results.  The fibration is now the Teichm\"u{}ller curve, denoted $\cX$ above; $\cX$ is the natural fiber space over Teichm\"u{}ller space, $\mathcal T = \cT_g$.  The fiber, $\cX_t$, for $t \in \cT$ is the Riemann surface $X=X_t$.  The line bundle $\mathcal L=\mathcal K^{m-1}$ whose restriction on each surface $X$ is $\mathcal K_X^{m-1}$.  The metric, indicated by $e^{-\phi}$ in \cite{Bo-mz} on $\mathcal K$ is in our case $\rho^{-1}= y^{2}$ on each fiber $X$.  More precisely, on each fiber, $X$ the metric is $\rho(z) |dz|^2$, with $\rho(z) = y^{-2}$, defined via the identification of $X = \Gamma \backslash \HH$.  The harmonic Beltrami differential $\mu$ here is a representative of the element $k_t$ in the Kodaira-Spencer class as described above.  The first formula  (\ref{eq:bob-1}) for $m=1$ is now an immediate consequence of \cite[Theorem 1.1]{Bo-mz}. 

To state the curvature formula in \cite[Theorem 1.2]{Bo-mz} 
for $m\ge 2$ we recall that the metric on $\mathcal K^{m-1}$ is $e^{-\psi} = y^{2m-2}$, with $\psi=(m-1)\phi$.   The complex gradient $V_{\psi}$, of $\psi$  with respect to the
fixed potential $\phi$ (see \cite[p. 1213]{Bo-mz}) 
can  be chosen as $V_{\psi}=V_{\phi}$.  
The corresponding Kodaira-Spencer class $k_t ^{(m-1)\phi}$
is represented by $\bar \partial V_{\phi}$ restricted to the fiber space and is thus also $k_t$,  so that 
$$k_t ^{(m-1)\phi}=k_t.$$ 
We put 
 $\eta=-k_t\cdot u=-\mu\cdot u$, a $(m-1, 1)$-form.
The curvature formula
in 
\cite[p. 1214--1215]{Bo-mz}
reads as follows
$$\langle R^{(m)}(\mu, \mu)u, u\rangle =\int_X c((m-1)\phi) |u|^2 e^{-\phi}  + \Vert \eta\Vert^2 -\Vert \xi\Vert^2.$$
Here, $\xi$ is the $L^2$-minimal solution of the $\bar\partial$ equation, 
$$\bar\partial \xi =-\nabla'\eta.$$

The function $c((m-1)\phi)$ is linear in $m-1$ by its definition \cite[(1.2)]{Bo-mz}, and so we have $c((m-1)\phi)=(m-1)c(\phi)$. Moreover, as explained in \cite[p. 1217]{Bo-mz}, it follows from
Schumacher's formula \cite[Proposition 3]{Schumacher-12} that 
$$c(\phi) = (1+\square_0)^{-1} |\mu|^2= f(\mu).$$ 
 The equation for $\xi$ is now solved by 
$$
\xi=-({\bar\partial}^\ast  {\bar\partial})^{-1}
{\bar\partial}^\ast\nabla' \eta
=
-{\bar\partial}^\ast
({\bar\partial}{\bar\partial}^\ast )^{-1}
\nabla' \eta.
$$
Thus 
\begin{equation*} \Vert  \xi \Vert^2 =-\langle  {\bar\partial}^\ast ({\bar\partial}{\bar\partial}^\ast )^{-1} \nabla' \eta, \xi\rangle
= \langle  ({\bar\partial}{\bar\partial}^\ast )^{-1} \nabla'  \eta, \nabla' \eta\rangle
\end{equation*}
Now by Lemma \ref{lemma1} we have ${\bar\partial}{\bar\partial}^\ast=
\nabla' (\nabla')^\ast +\frac{m-1}2
$ on the $(m, 1)$ form $\nabla' \eta$ and 
\begin{equation*}
({\bar\partial}{\bar\partial}^\ast)^{-1} \nabla' 
=\left(\nabla' (\nabla')^\ast +\frac{m-1}2\right)^{-1}  \nabla' =\nabla' 
\left(
(\nabla')^\ast\nabla'  +\frac{m-1}2\right)^{-1}
\end{equation*}
on the $(m-1, 1)$ form $\eta$.
Hence
$$
\Vert \xi \Vert^2
=\langle \nabla' 
\left((\nabla')^\ast\nabla'  +\frac{m-1}2\right)^{-1}\eta, \nabla'\eta\rangle
=
\langle \left((\nabla')^\ast\nabla'  +\frac{m-1}2\right)^{-1}\eta, 
(\nabla')^\ast\nabla'\eta\rangle
$$ 
and  finally
$$
\Vert \eta\Vert^2 -\Vert \xi\Vert^2
=\frac{m-1}2\langle \left(
(\nabla')^\ast\nabla' +\frac{m-1}2\right)^{-1}\eta, \eta\rangle
=\frac{m-1}2\langle \left(
\square'' +\frac{m-1}2\right)^{-1}\eta, \eta\rangle,
$$
using again Lemma \ref{lemma1}. This completes the proof.
\end{proof}

\begin{remark} Since the sum $\sum_j |u|_j^2$ in $I^{(m)}$ is the Bergman kernel, its expansion \cite{Lu-AJM} for large $m \in \N$ could be used to compute the expansion of  $\cher^{(m)}(\mu, \mu)$. Here we instead compute the explicit formula for the second variation of the Selberg zeta function because it contains much detailed geometric data.  In this way we obtain the asymptotic expansion of $\cher^{(m)}(\mu, \mu)$ in $m$ as $m \to \infty$ as a corollary.  We note that in \cite{Sun} the second term  $\frac{m-1}2\langle (\square'' +\frac{m-1}2)^{-1}\eta, \eta\rangle$ above appears as  $({m-1})\langle (\Delta +{m-1})^{-1}\eta, \eta\rangle$ 
where $\Delta $ is the $\bar\partial$-Laplacian, $\square''$.
The discrepancy with our formula is due to our definition of the norms on the $L^2$-spaces of $(m, 1)$ forms, so with this consideration, our formulas agree.  
\end{remark}

\subsection{Zeta-regularized determinant, analytic torsion, Ruelle and higher zeta functions} \label{s-zdet} 
Let $\det(\Delta_0 +s(s-1))$ be the zeta-regularized determinant of the Laplacian on scalars.  This is defined through the spectral zeta function.  For $\Re(s) >1$ and $\Re(z) >1$, this spectral zeta function is defined by 
$$\zeta (z) = \sum_{k \in \N} (\lambda_k + s(s-1))^{-z},$$
where $\{ \lambda_k \}_{k \in \N}$ are the eigenvalues of $\Delta_0$.  In the special case $s=1$, the sum above is only taken over the non-zero eigenvalues of $\Delta_0$.  It is well known that the spectral zeta function admits a meromorphic extension to $z \in \C$ which is regular at $z=0$.  The determinant is then defined to be $\exp( - \zeta'(0))$.  

This determinant is closely related to the Selberg zeta function (\ref{SelbergZeta}).  By \cite[Theorem 1]{Sarnak-cmp-87} the determinant,  $\det(\Delta_0 +s(s-1))$, and the Selberg zeta function, $Z(s)$, are related by 
\begin{equation} \label{sarnakzdet}  \det(\Delta_0 +s(s-1))=Z(s)  \left( e^{E-s(s-1)}  \frac {\Gamma_2(s)^2 }{\Gamma(s)}  (2\pi)^{s} \right)^{2g-2}, \quad \Re(s) > 1. \end{equation} 
Above, $g$ is the genus, $\Gamma_2 (s)$ is the Barnes double gamma function, defined by the canonical product 
$$\frac{1}{\Gamma_2 (s+1)} = (2\pi)^{s/2} e^{-s/2 - \frac{\gamma+1}{2}s^2} \prod_{k=1} ^\infty \left( 1 + \frac s k \right)^k e^{-s + s^2/2k}, \textrm{ $\gamma$ is Euler's constant,}$$  
and 
$$E = - \frac 1 4 - \frac 1 2 \log(2\pi) + 2\left( \frac{1}{12} - \log(A) \right).$$ 
Above, $A$ is the Glaisher-Kinkelin constant.\footnote{The term with $\frac{1}{12}-\log(A)$ comes from the derivative of the Riemann zeta function at $-1$.} 
In the special case $s=1$, we have 
\begin{equation} \label{sarnakzdet0} \det(\Delta_0) =  Z'(1) e^{(2g-2)(-1/(12)  -2 \log(A) + (\log(2\pi))/2)}. \end{equation} 

\subsubsection{Holomorphic analytic torsion}
Holomorphic analytic torsion, or $\bar{\partial}$-torsion, as Ray \& Singer originally introduced  in their pioneering work \cite{RaySinger73}, is a complex analogue of analytic torsion.  To define it, let $\mathcal{D}^{p,q}$ be the set of $C^\infty$ complex $(p,q)$-forms on $X$.  Then the exterior differential $d$ splits as 
$$d = d' + d'',$$
where 
$$d' : \mathcal{D}^{p,q} \mapsto \mathcal{D}^{p+1,q},$$
$$d'' : \mathcal{D}^{p,q} \mapsto \mathcal{D}^{p,q+1}.$$
Let $D_{p,q}$ be the corresponding Laplacian,
$$D_{p,q} = * d'*d'' + d''*d'*: \mathcal{D}^{p,q} \mapsto \mathcal{D}^{p,q}.$$
We note that defined in this way, as in \cite{RaySinger73}, the eigenvalues of this operator are non-positive.  Thus, they defined the associated spectral zeta function,  
\begin{equation*}
\zeta_{D_{p,q}} (s)= \sum_{\lambda_n \neq 0} (- \lambda_n)^{-s}
\end{equation*}
for $\Re (s)$ large.  This zeta function may also be expressed in terms of the Mellin transform of the heat trace.  In this way, using the short time asymptotic expansion of the heat trace, one can prove that $\zeta_{D_{p,q}}$ extends to a meromorphic function in $\C$ which is regular at $s=0$.  One may therefore make the following 

\begin{defn}
For each $p = 0, \ldots, N$, where $N$ is the complex dimension of $X$,  the holomorphic analytic torsion, $T_p(X)$, is defined by
$$ \log T_p(X) = \frac{1}{2} \sum_{q=0}^N (-1)^q q \zeta'_{D_{p,q}}(0).$$
\end{defn}

In our case, $N=1$, and so there are two holomorphic analytic torsions, 
\begin{equation} \label{torsions} T_0 (X) = e^{ - \frac{1}{2}\zeta'_{D_{0,1}}(0)} \quad \textrm{and} \quad T_1 (X) = e^{- \frac{1}{2} \zeta'_{D_{1,1}}(0)}. \end{equation} 
It is well known that the non-zero eigenvalues of $D_{0,1}$ coincide with those of $D_{0,0} = \Delta_0$ as do those of $D_{1,1}$ \cite{Berth}.  We therefore have 
\begin{equation} \label{atordet}  T_0 (X) = \sqrt{ \det (\Delta_0)} = T_1 (X). \end{equation} 

\subsubsection{Ruelle and higher zeta functions} 
We shall also consider the Ruelle zeta function.   For $\Re(s) > 1$, the Ruelle zeta function is 
\begin{equation} \label{ruellezeta} 
R(s)
 := \prod_{\gamma \in \prim} \( 1 - e^{-s l(\gamma)}  \).
\end{equation}
Above, $\prim$ denotes the set of conjugacy classes of primitive hyperbolic elements in $\Gamma$.  

The last type of zeta functions which are relevant to our present work are the higher Selberg zeta functions, the first of which was introduced and studied by Kurokawa and Wakayama \cite{Kurokawa2004}:
\begin{equation}\label{higherSelbergZeta}
z(s) := \prod_{\gamma \in \prim} z_\gamma (s), \quad z_\gamma(s) :=  \prod_{k=1}^\infty \(1 - e^{-l(\gamma)\cdot (s+k)}\)^{-k}.	\end{equation}
More generally, this notion of higher Selberg zeta function was generalized in \cite{Hashimoto2005} who defined a whole procession of higher Selberg zeta functions. For $t \in \C$ and $\Re(s) > 1$, let 
\begin{equation}\label{hierSelbergZeta}
z(s, t) := \prod_{\gamma \in \prim} z_\gamma(s,t), \quad z_\gamma(s,t)  :=  \prod_{k=0}^\infty \(1 - e^{-l(\gamma) \cdot (s+k)}\)^{{\binom{t+k-1}{k}} }.
\end{equation}

\section{First variation}
In this section we compute the first variation of $\log Z(s)$. We start by computing the variation of the length of an individual closed geodesic, $\partial_\mu l(\gamma)$.  

\subsection{Variation of the length of a closed geodesic}  \label{s:vargeo} 
We shall demonstrate a variational formula for the length of a closed geodesic in Proposition \ref{lengthvarprop1} using \cite[Theorem 1.1]{Ax-Sch-2012}.  There, Axelsson \& Schumacher worked in the more general context of families of K\"ahler-Einstein manifolds.  Here, we note that by \cite[Theorem 5]{Ahlfors-1961}, for a neighborhood (with respect to the Weil-Petersson metric) in Teichm\"uller space, the corresponding family of Riemann surfaces, each equipped with the unique hyperbolic Riemannian metric of constant curvature $-1$, form a holomorphic family.  Thus, we are indeed in the setting of \cite{Ax-Sch-2012}.  Related works include \cite{Ax-Sch-2010,  wolpert1981}, and \cite{Gon}.  

Here we shall prove our variational formula, Proposition \ref{firstvarlogszf} using Gardiner's formula, \cite[Theorem 8.3]{Im-Ta}.  This formula states that the variation $\partial_\mu l(\gamma)$ is given by\footnote{We note that in \cite[Theorem 8.3]{Im-Ta}, they  considered the real variation, so they have $\Re \l(\mu,\frac 2\pi \Theta_\gamma \)$.  We are taking the complex variation, and so we have the formula above for the variation.} 
\begin{equation}\label{Gardiner}
\partial_\mu l(\gamma) =\l(\mu,\frac 1\pi \Theta_\gamma \)
\end{equation}
where $\Theta_{\gamma}$ is \cite[p. 224]{Im-Ta} 
$$\Theta_{\gamma}=\sum_{\kappa\in \la\gamma\ra \backslash \Gamma}(\omega_\gamma \circ \kappa) \cdot (\kappa')^2. $$
Above, $\omega_\gamma(z)=\(\frac {a-b}{(z-a)(z-b)}\)^2$, where $a, b \in \partial \HH$ are the fixed points of $ \gamma $, $\langle \gamma \rangle$ is the cyclic subgroup of $\Gamma$ generated by $\gamma$, and $\kappa'$ denotes the derivative of $\kappa$. We also recall (c.f. \cite{Gon}) that 
for any $\gamma\in \Gamma$, the vector field 
\begin{equation}\label{qzfield} q_\gamma(z)= (cz^2 + (d-a)z -b) \frac {\partial}{\partial z}
\end{equation}
satisfies 
$$ (\sigma^{-1})_\ast q_\gamma(z)= q_{\sigma^{-1}\gamma\sigma}(z), \quad \forall \, \sigma \in PSL(2, \mathbb R).$$
In particular, the integral
\begin{equation}\label{peridointegraldef}
\int_{z_0}^{\gamma z_0} \varphi(z) (dz)^2q_\gamma(z) = \int_{z_0} ^{\gamma z_0} \varphi(z)  (cz^2 + (d-a)z -b) dz.
\end{equation}
is independent of both the path and the starting point $z_0$ for any weight 4 holomorphic modular form, $\varphi$ for $\Gamma$.  We may therefore use this integral to define the period integral as in  \cite[Definition 1.1.1]{Gon}.
\begin{defn}
For any weight 4 holomorphic modular form $\varphi$ for $\Gamma$ and a hyperbolic element $\gamma \in \Gamma$, we define the period integral $\alpha(\gamma, \varphi)$ by (\ref{peridointegraldef}).
\end{defn} 

\begin{prop} \label{lengthvarprop1} 
The variation of the length $l(\gamma)$ is given by
$$ \frac{\partial}{\partial \mu} l(\gamma)=\frac 12 \int_{\gamma} \mu(z(t))\overline{\dot{z}(t)}^2 \rho(z(t)) dt
= -\frac{\overline{\alpha(\gamma, \phi)}}{ 4 \sinh(l(\gamma)/2)}.$$
Above, $z=z(t)$ is a parametrization of the geodesic representing $\gamma$ with 
$\Vert \dot z(t)\Vert=1$, and the integration is over the footprint of $\gamma$ in $X$.  We note that the variation is independent of the choice of parametrization, and $\phi$ is defined via $\mu$ as in (\ref{phi-mu}).  
\end{prop}

\begin{proof}
We can assume $\gamma$ is the diagonal matrix $\gamma=\text{diag}(e^{\frac l 2},  e^{-\frac l 2})$ with $l=l(\gamma)$.  We choose 
$$F_0=\{z\in \HH; 1<|z|<e^{l}\}$$ 
as a fundamental domain of the cyclic subgroup $\langle\gamma\rangle$. Gardiner's formula  (\ref{Gardiner}) can now be written as (see \cite[Theorem 8.3, pp. 226-227]{Im-Ta}) 
$$ \partial_\mu l(\gamma) =\frac 1\pi  \int_{F_0}\frac {\mu(z)}{z^2 } dx dy =\frac 1\pi \int_{F_0} \frac{ {\bar \phi(z) y^2}}{{z^2}}  dx dy.$$
Writing in polar coordinates
$z=e^{t}e^{i\theta}$, $0 < \theta<\pi$, $0 \leq t \leq l$,  $dxdy=e^{2t} dt d\theta$,
 we have
\begin{equation}\label{partialbarlgamma}
\begin{gathered}
\overline{\partial_\mu l(\gamma)}
=\frac 1\pi \int_{F_0}\frac {\phi(z)y^2}{(\bar z)^2} dxdy
=\frac 1\pi \int_0^\pi \int_{0}^{l}\phi(e^t e^{i\theta}) \frac{(e^t\Im e^{i\theta})^2}{ (e^t e^{-i\theta})^2} e^{2t}
dt d\theta \\
=-\frac 1\pi  \frac 14
\int_0^\pi 
\int_{0}^{l}\phi(e^t e^{i\theta}) (1-e^{2i\theta})^2 e^{2t}
dt d\theta.
\end{gathered}
\end{equation}

Consider the segment, $R_\theta$, defined by $z(t)=e^t e^{i\theta}$, $0 \leq t \leq l$.  Then the vector field $q(z)=q_\gamma(z)$ from (\ref{qzfield}) is 
$$q(z)= - 2 \( \sinh  \frac {l(\gamma)}{2} \) \frac{d}{dt}$$
and the quadratic form 
$$\phi(z) (dz)^2= \phi(e^t e^{i\theta}) e^{2t} e^{2i \theta}  (dt)^2.$$ 
Hence,
\begin{equation} \label{alphaC} \alpha(\gamma, \phi) = \int_{R_\theta} \phi(z) (dz)^2q(z) = -2 \(\sinh\frac {l(\gamma)}{ 2} \) \int_{0}^l  \phi(e^t e^{i\theta}) e^{2t} e^{2i \theta} dt. \end{equation}  
With the small observation that $(\dot z(t))^2 = (z(t))^2$, we also have 
\begin{equation} \label{alphaC3} - \frac{ \alpha(\gamma, \phi)}{4 \(\sinh\frac {l(\gamma)}{ 2} \) }  = \frac1 2 \int_\gamma \phi(z(t)) (\dot z(t))^2 dt = \frac 1 2 \int_{\gamma} \overline{\mu(z(t))} (\dot{z}(t))^2 \rho(z(t)) dt. \end{equation} 
Note that the integral in (\ref{alphaC}) does not depend on the path nor the starting point.  Consequently, for any $\theta \in (0, \pi)$, 
$$ \int_{0}^l  \phi(e^t e^{i\theta}) e^{2t} e^{2i \theta} dt = - \int_{0}^l  \phi(e^t i) e^{2t} dt =: C, $$
and 
\begin{equation} \label{alphaC2} \alpha(\gamma, \phi) =  - 2 \(\sinh\frac {l(\gamma)}{ 2} \) C. \end{equation} 
Substituting $C$ into (\ref{partialbarlgamma}), we obtain
$$\overline{\partial_\mu l(\gamma)} =-\frac 1\pi \frac 14 C \int_0^\pi 
(1-e^{2i\theta})^2 e^{-2i \theta}d\theta = -\frac 1\pi  \frac 14 C (-2\pi)= \frac 1 2 C. $$
By (\ref{partialbarlgamma}),  (\ref{alphaC2}), and (\ref{alphaC3})
$$  \overline{\partial_\mu l(\gamma)} = \frac 1 2 C = - \frac{ \alpha(\gamma, \phi)} {4 \sinh(l(\gamma)/2)} = \frac 12 \int_{\gamma} \overline{\phi(z(t))}\,\overline{\dot{z}(t)}^2 dt = \frac 12 \int_{\gamma} \mu(z(t)) \overline{\dot{z}(t)}^2 \rho(z(t)) dt.
$$
\end{proof}

\begin{remark}  Our approach to compute the variation of the length may be compared to Wolpert's in \cite{wolpert86}.  \end{remark}

We also recall the second variation formula of Axelsson-Schumacher \cite[Theorem~6.2]{Ax-Sch-2012}.

\begin{prop} \label{lengthvarprop2} 
The variation $\bar\partial_\mu\partial_\mu l$ of $l=l(\gamma)$ is given by
$$\bar\partial_\mu\partial_\mu l =\frac 12 \int_{\gamma}\left( (\square_0 +1)^{-1}(|\mu|^2) + \(-\frac{D^2}{dt^2} +2\)^{-1} (\mu) \bar \mu \right)  + \frac{1}{l}|\partial_\mu l|^2. $$
Here $\frac{D}{dt}$ is differentiation along the geodesic $\gamma$.   
\end{prop}

An immediate consequence of  \cite[Corollary 1.7]{Ax-Sch-2012} is 
\begin{cor} \label{cor2} 
There exists a constant $C>0$ which is independent of $\gamma$ 
such that 
$$|\partial_\mu l(\gamma) |\le C \Vert \mu\Vert_{\infty} l(\gamma),\quad \textrm{and} \quad |\bar\partial_\mu\partial_\mu l (\gamma) 
|\le C  \Vert \mu\Vert_{\infty}^2 l (\gamma).$$
\end{cor}

\subsection{Convergence}
In order to rigorously justify our variational formulas we must demonstrate the convergence of a certain sum.  

\begin{lemma}\label{convergence_lemma}
For $n \in \N$, the sum 
\begin{equation*}
	\sum_{\gamma \in \prim} \partial_\mu (l(\gamma))   \sum_{k \ge 1} \frac{k^n \cdot l(\gamma)}{e^{l(\gamma)\cdot (s+k)}-1}
	\end{equation*}
	converges absolutely and uniformly on compact subsets of $\Re(s) > 1$.  
\end{lemma}

\begin{proof}
Let $\varepsilon> 0$ satisfy 
	\begin{equation}\label{conv_est1}
	k^n < e^{ \varepsilon k}, \quad k \in \N \backslash \{0 \}.
	\end{equation} 
	First, we consider only geodesics whose lengths, $L$ satisfy $L > \max(1,\varepsilon)$. This implies $e^{l(\gamma)\cdot (\Re(s)+k) - 1} > 1$ as $\Re(s)+k \ge 1$, and hence 
	\begin{equation*}
	e^{l(\gamma) \cdot (\Re(s)+k) - 1}( e - 1 ) > 1 
	\end{equation*}
	and 
	\begin{equation}\label{conv_est2}
	e^{l(\gamma) \cdot (\Re(s)+k)} - 1>  e^{l(\gamma) \cdot (\Re(s)+k)-1}.
	\end{equation}
	It follows from (\ref{conv_est1}) and (\ref{conv_est2}) that
	\begin{equation}\label{conv_est4}
	\frac{k^n}{|e^{l(\gamma)\cdot(s+k)}-1|} \le \frac{k^n}{e^{l(\gamma)\cdot(\Re(s)+k)}-1} \le \frac{e}{e^{l(\gamma)\cdot (k+\Re(s))-\varepsilon\cdot k}}.
	\end{equation}
	Note that
	\begin{equation}\label{conv_est3}
	\sum_{k \ge 1} \frac{1}{e^{l(\gamma)\cdot (k+\Re(s))-\varepsilon\cdot k}} = \frac{1}{e^{l(\gamma) \cdot \Re(s)} } \sum_{k \ge 1} \frac{1}{e^{k \cdot (l(\gamma)-\epsilon)}} \le  \frac{1}{e^{l(\gamma) \cdot \Re(s) } } \cdot  \frac{1}{e^{l(\gamma)-\epsilon}-1} . 
	\end{equation}
	The geometric progression converges because $l(\gamma) - \varepsilon > 0$,  hence $e^{l(\gamma)-\varepsilon} > 1$. Combining (\ref{conv_est3}) and (\ref{conv_est4}) with $|\partial_\mu l(\gamma)| \le \frac{1}{2} \| \mu\|\cdot l(\gamma)$, we obtain
$$\sum_{\gamma \in \prim, l(\gamma) \ge L} 
| \partial_\mu (l(\gamma))|    \sum_{k \ge 0} \left|\frac{k^n \cdot l(\gamma)}{e^{l(\gamma)\cdot (s+k)}-1}\right| \le  e \| \mu\| \cdot \sum_{\gamma \in \prim, l(\gamma) \ge L}  \frac{1}{e^{l(\gamma)-\epsilon}-1} \frac{l(\gamma)^2}{e^{l(\gamma) \cdot \Re(s) }}$$ 

$$\leq \frac{1}{e^{l_0 - \epsilon} -1} e \| \mu\|_\infty \cdot \sum_{\gamma \in \prim, l(\gamma) \ge L}\frac{l(\gamma)^2}{e^{l(\gamma) \cdot \Re(s) }},$$
where $l_0$ denotes the length of the shortest closed geodesic.   
	Recall that  \cite[p. 33]{borthwick2007spectral},
	\begin{equation}
	\sum_{\gamma \in \prim} e^{- l(\gamma) s_1}
	\end{equation}
	converges absolutely for $\Re(s_1) > \delta=1$. Choose $\varepsilon_1(L) > 0$ such that $l(\gamma)^2 \le e^{l(\gamma) \cdot \varepsilon_1(L)}$ for  $l(\gamma) \ge L$.
	\begin{equation}
	\sum_{\gamma \in \prim, l(\gamma) \ge 1} \frac{l(\gamma)^2}{e^{l(\gamma) \cdot \Re(s) }} \le \sum_{\gamma \in \prim} \frac{1}{e^{l(\gamma) \cdot (\Re(s) - \varepsilon_1(L))}}
	\end{equation}
	converges for $\Re(s) > \delta + \varepsilon_1(L)$.
	We are left with $l(\gamma) < L$. As there are only finite many such geodesics, it is obvious that the series
	\begin{equation}
	\sum_{\gamma \in \prim,\,  l(\gamma) < L} \sum_{k \ge 0} \frac{k^n \cdot l(\gamma)}{e^{l(\gamma)\cdot (s+k)}-1}
	\end{equation}
	converges absolutely for any fixed $s \in \C$. 
	
	The last step is to consequently take $L$ larger and larger, allowing us to make  $\varepsilon_1(L)$ arbitrary small.  This shows that the series converges absolutely and uniformly on compact subsets of $\Re(s) > 1$.  
\end{proof}

\subsection{First variational formula of $Z(s)$ for $\Re s>1$} 
We introduce the local zeta function as in \cite{Gon}, 
\begin{equation} \label{localzeta} 
z_\gamma(s):=\prod_{k\ge 1}
\left(1-e^{-l(\gamma) (s+k)} \right)^{-k}
\end{equation} 
We also recall the local Selberg zeta function for convenience here, 
\begin{equation} \label{localszf} Z_\gamma(s) = \prod_{k \in \N} \left( 1- e^{l(\gamma) (s+k)} \right). \end{equation} 
We define further the higher local zeta function, 
\begin{equation} \label{localhighzeta} 
\tilde z_\gamma(s):=\prod_{k\ge 1}
\left(1-e^{-l(\gamma) (s+k)} \right)^{-k^2}
\end{equation} 
We now have requisite tools to prove the first variational formula.  

\begin{prop} \label{firstvarlogszf} 
Suppose $\Re(s) >1$.
$$
\partial_\mu \log Z(s)
=\sum_{
\gamma \in \prim}
\partial_\mu \log l(\gamma)\(
s\frac {d}{ds}\log Z_{\gamma}(s)
+\frac {d}{ds}\log z_{\gamma}(s)^{-1}\).
$$
Above, $\frac{\partial \log l(\gamma)}{\partial \mu}
=l(\gamma)^{-1} \frac{\partial l(\gamma)}{\partial \mu}$
and  $\frac{\partial l(\gamma)}{\partial \mu}$ is given by Propositions \ref{lengthvarprop1} and \ref{lengthvarprop2}.
\end{prop}

\begin{proof} We compute by definition of $Z(s)$ and Lemma \ref{convergence_lemma} which allows us to differentiate termwise,  
$$\pa_\mu \log Z_\gamma(s) = \sum_{\gamma \in \prim} \sum_{k \in \N} \pa_\mu \left( \log(1-e^{-l(\gamma) (s+k)}) \right)= \sum_{\gamma \in \prim} \sum_{k \in \N} \frac{ \pa_\mu (l(\gamma)) (s+k) e^{-l(\gamma) (s+k)}}{1-e^{-l(\gamma)(s+k)}}$$
$$= \sum_{\gamma \in \prim} \sum_{k \in \N}  \partial_\mu l(\gamma) \frac{s+k}{e^{l(\gamma)(s+k)}-1}.$$
By Lemma \ref{convergence_lemma}, this converges absolutely and uniformly on compact subsets of $\Re(s)~>~1$.  
Recalling the definition of the local Selberg zeta function (\ref{localszf}), and the local zeta function (\ref{localzeta}) we compute in a similar fashion that 
$$\frac{d}{ds} \log Z_\gamma (s) = \sum_{k \in \N} \frac{ l(\gamma)}{e^{l(\gamma)(s+k)} - 1} \textrm{ and }  \frac{d}{ds} \log z_\gamma (s)^{-1} = \sum_{k \in \N} \frac {k}{e^{l(\gamma) (s+k)}-1}.$$
Therefore we have 
\begin{equation} \label{sdds} s \frac{d}{ds} \log Z_\gamma(s) = \sum_{k \in \N} \frac{ s l(\gamma)}{e^{l(\gamma)(s+k)} - 1} \textrm{ and } \frac{d}{ds}  \log z_\gamma (s)^{-1} = \sum_{k \in \N} \frac {k}{e^{l(\gamma) (s+k)}-1}. \end{equation} 
Hence, 
$$\pa_\mu \log Z_\gamma(s) = \sum_{\gamma \in \prim} \frac{\pa_\mu l(\gamma)}{l(\gamma)} \left( s \frac{d}{ds} \log Z_\gamma(s) + \frac{d}{ds}  \log z_\gamma (s)^{-1} \right) $$
$$= \sum_{\gamma \in \prim} \pa_\mu (\log l(\gamma)) \left( s \frac{d}{ds} \log Z_\gamma(s) + \frac{d}{ds}  \log z_\gamma (s)^{-1} \right).$$
\end{proof}

\section{Second variation} 

We prove here the second variational formula for the log of the Selberg zeta function as well as first and second variational formulas for:  the Ruelle zeta function, the zeta-regularized determinant of the Laplacian, the square of the Hilbert-Schmidt norm of the resolvent, and the higher zeta functions, 

\subsection{Second variation of Selberg zeta function} 
In preparation for the proof, we note first the following differentiation formulas
$$\frac{d^2}{ds^2}\log Z_{\gamma}(s) = - \ell(\gamma)^2 \sum_{k=0}^\infty \frac{e^{ \ell(\gamma) (s+k)}}
{(e^{\ell(\gamma)(s+k)}-1)^2},
$$

$$
\frac{d^2}{ds^2}
\log z_\gamma (s)^{-1}
=
- \ell(\gamma)^2 \sum_{k=0}^\infty
\frac{k
e^{\ell(\gamma)(s+k)}
 }
{(e^{\ell(\gamma)(s+k)}-1)^2
},
$$
$$
\frac{d^2}{ds^2}
\log \tilde z_\gamma (s)^{-1}
=
- \ell(\gamma)^2 \sum_{k=0}^\infty
\frac{
k^2 
e^{\ell(\gamma)(s+k)}
}
{
(e^{\ell(\gamma)(s+k)} -1)^2
}.
$$

Hence, we see that 
\begin{multline} \label{secondvarwarmup} s^2 \frac{d^2}{ds^2} \log Z_\gamma(s) + 2s \frac{d^2}{ds^2} \log z_\gamma(s) + \frac{d^2}{ds^2} \log \tilde z_\gamma(s) \\ = - \ell(\gamma)^2 \sum_{k \in \N} \frac{(s+k)^2 e^{\ell(\gamma)(s+k)}}{(e^{\ell(\gamma)(s+k)} -1)^2}. \end{multline} 
 
\begin{proof}[Proof of Theorem \ref{second_variation_geod_sum}] 
We perform the differentiation $\bar\partial_\mu$ on the result obtained in Proposition \ref{firstvarlogszf}, 
$$\bar\partial_\mu\partial_\mu \log Z(s) = \bar \partial_\mu \left( \sum_{\gamma \in \prim} \partial_\mu \log \ell(\gamma)\(s\frac {d}{ds}\log Z_{\gamma}(s) +\frac {d}{ds}\log z_{\gamma}(s)^{-1}\) \right).$$
By Leibniz's rule, this is 
$$\sum_{\gamma\in \prim} \bar \pa_\mu \pa_\mu \log \ell(\gamma) \left( s\frac {d}{ds}\log Z_{\gamma}(s) +\frac {d}{ds}\log z_{\gamma}(s)^{-1}) \right) $$
$$+ \sum_{\gamma \in \prim} \pa_\mu \log \ell(\gamma) \bar \pa_\mu \left(  s \frac{d}{ds} \log Z_\gamma (s) + \frac{d}{ds} \log z_\gamma(s)^{-1} \right)$$
$$= \sum_{\gamma \in \prim} \bar\partial_\mu\partial_\mu \log \ell(\gamma) A_\gamma(s) + \sum_{\gamma \in \prim} \pa_\mu \log \ell(\gamma) \bar \pa_\mu \left(  s \frac{d}{ds} \log Z_\gamma (s) + \frac{d}{ds} \log z_\gamma(s)^{-1} \right).$$
Thus, it only remains to consider the second term.  For this, we compute 
$$\bar \pa_\mu \left(  s \frac{d}{ds} \log Z_\gamma (s) + \frac{d}{ds} \log z_\gamma(s)^{-1} \right) = \bar \pa_\mu A_\gamma(s)$$
$$= \bar \pa_\mu \sum_{k \in \N} \frac{(s+k)
  \ell(\gamma)}{e^{\ell(\gamma)(s+k)} -1} = \sum_{k \in \N} 
\left( \frac{\bar \pa_\mu \ell(\gamma) (s+k)}{e^{\ell(\gamma)(s+k)} -1} - \frac{ \ell(\gamma) (s+k)^2 \bar \pa_\mu \ell(\gamma) e^{\ell(\gamma) (s+k)}}{(e^{\ell(\gamma) (s+k)} -1)^2}\right).$$
Thus, we have for this second term, 
$$ \bar \pa_\mu A_\gamma(s) = \bar \pa_\mu \ell(\gamma) \sum_{k \in \N} \frac{s+k}{e^{\ell(\gamma)(s+k)} -1} -  \frac{ \ell(\gamma) (s+k)^2 e^{\ell(\gamma)(s+k)}}{(e^{\ell(\gamma)(s+k)} -1)^2}.$$
Hence the expression becomes 
$$\bar\partial_\mu\partial_\mu \log Z(s) = \sum_{\gamma \in \prim} \bar\partial_\mu\partial_\mu \log \ell(\gamma) A_\gamma(s) $$
$$+ \sum_{\gamma \in \prim} \pa_\mu \log \ell(\gamma) \bar \pa_\mu \ell(\gamma) \sum_{k \in \N} \frac{s+k}{e^{\ell(\gamma)(s+k)} -1} -  \frac{ \ell(\gamma) (s+k)^2 e^{\ell(\gamma)(s+k)}}{(e^{\ell(\gamma)(s+k)} -1)^2}.$$
The second term above is
$$\sum_{\gamma \in \prim} \pa_\mu \log \ell(\gamma) \bar \pa_\mu
\ell(\gamma) \sum_{k \in \N} \left(
\frac{s+k}{e^{\ell(\gamma)(s+k)}
    -1} -  \frac{ \ell(\gamma)(s+k)^2
    e^{\ell(\gamma)(s+k)}}{(e^{\ell(\gamma)(s+k)} -1)^2}\right)
$$
$$= \sum_{\gamma \in \prim} |\pa_\mu \log \ell(\gamma)|^2 \sum_{k \in
  \N} \left(\frac{(s+k) \ell(\gamma)}{e^{\ell(\gamma)(s+k)} -1} -
  \frac{ \ell(\gamma)^2 (s+k)^2
    e^{\ell(\gamma)(s+k)}}{(e^{\ell(\gamma)(s+k)} -1)^2}
\right)
$$
$$= \sum_{\gamma \in \prim} |\pa_\mu \log \ell(\gamma)|^2 A_\gamma(s)
-  \sum_{\gamma \in \prim} |\pa_\mu \log \ell(\gamma)|^2
\ell(\gamma)^2 \sum_{k \in \N}  
\left(\frac{ (s+k)^2
  e^{\ell(\gamma)(s+k)}}{(e^{\ell(\gamma)(s+k)} -1)^2}\right),$$
where we have used the simple fact 
$\pa_\mu \log \ell(\gamma) \bar \pa_\mu \ell(\gamma) = |\pa_\mu \log \ell(\gamma)|^2 \ell(\gamma).$
By our preliminary calculation (\ref{secondvarwarmup}), 
$$- \sum_{\gamma \in \prim} |\pa_\mu \log \ell(\gamma)|^2
\ell(\gamma)^2 \sum_{k \in \N}  \frac{(s+k)^2
  e^{\ell(\gamma)(s+k)}}{(e^{\ell(\gamma)(s+k)} -1)^2} 
$$ 
$$ =  \sum_{\gamma \in \prim}  |\pa_\mu \log \ell(\gamma)|^2 \left( s^2 \frac{d^2}{ds^2} \log Z_\gamma(s) + 2ks \frac{d^2}{ds^2} \log z_\gamma(s) + \frac{d^2}{ds^2} \log \tilde z_\gamma(s)\right)$$
$$= \sum_{\gamma \in \prim}  |\pa_\mu \log \ell(\gamma)|^2 B_\gamma (s).$$
Thus, the total expression is 
$$ \bar\partial_\mu\partial_\mu \log Z(s) = \sum_{\gamma \in \prim} \bar\partial_\mu\partial_\mu \log \ell(\gamma) A_\gamma(s) + |\pa_\mu \log \ell(\gamma)|^2 (A_\gamma(s) + B_\gamma(s)).$$
Finally, we note that convergence follows from the estimates in Corollary \ref{cor2} and Lemma \ref{convergence_lemma}.
\end{proof}
\subsection{Applications to further variational formulas} \label{randomvars} 
The following formula for $\tr (\Delta +s(s-1))^{-2}$ may be proven as a consequence of the Selberg trace formula
for the divided resolvents \cite{hejhal2006selberg}, \cite[p. 118]{Sarnak-cmp-87}.  We have a somewhat different proof which may be of independent interest; this proof comprises \S \ref{appendixa}.  

\subsubsection{Variation of the resolvent} 
\begin{lemma} \label{le:hsnorm} 
Let $\Re s>1$. The squared resolvent 
$(\Delta_0 +s(s-1))^{-2}$  is of trace class, and the Hilbert-Schmidt norm of the resolvent is given by its trace, 
\begin{equation*}\begin{gathered} ||(\Delta_0 + s(s-1))^{-1}||^2 _{HS} = \tr (\Delta_0 +s(s-1))^{-2} = \sum_{k \in \N} \frac{1}{(\lambda_k + s(s-1))^2} 
\\ = \frac{(g-1)\pi^2}{\alpha} (1+\tan^2(\pi \alpha)) +  \frac{(g-1)}{4 \alpha s^2} - \frac{(g-1)s}{\alpha} \sum_{n \geq 1} \frac{ n}{(n^2 - s^2)^2} - L^2 \log Z(s),   \end{gathered} \end{equation*}
where above $\alpha = s-1/2$, and the operator $L := \frac{1}{2s-1} \frac{d}{ds}$. 
\end{lemma}

Lemma \ref{le:hsnorm} and Lemma \ref{convergence_lemma} immediately imply 
\begin{cor}
For $\Re(s) > 1$, the first and second variations of the square of the Hilbert-Schmidt norm of the resolvent are respectively 
$$\pa_\mu || (\Delta_0 + s(s-1))^{-1} ||_{HS} ^2 = - L^2 \pa_\mu \log Z(s),$$
$$\bar \pa_\mu \pa_\mu  || (\Delta_0 + s(s-1))^{-1} ||_{HS} ^2 = - L^2 \bar \pa_\mu \pa_\mu \log Z(s).$$
\end{cor} 

\subsubsection{Variation of the Ruelle zeta function} 	
\begin{prop}\label{Ruelle_variation} 
For $\Re(s) >1$, we have the first and second variations of the Ruelle zeta function 
$$
\partial_\mu \log R(s) 
= 
\sum_{\gamma \in \prim} \partial_\mu( l(\gamma))  
s e^{-s l(\gamma)} 
\( 1 - e^{-s l(\gamma)}   \)^{-1}
=\sum_{\gamma \in \prim} \partial_\mu( l(\gamma))  
s \( e^{s l(\gamma)} -1  \)^{-1}
$$
$$
\partial_\mu \bar\partial_\mu 
\log R(s) 
= 
\sum_{\gamma \in \prim}
\bar\partial_\mu \partial_\mu( l(\gamma))  
s\( 1 - e^{-s l(\gamma)}   \)^{-1} 
-\sum_{\gamma \in \prim} s^2
| \partial_\mu( l(\gamma))  |^2 e^{s l(\gamma)} 
\(  e^{s l(\gamma)} 
-1   \)^{-2} 
$$
\end{prop} 
\subsubsection{Variation of the determinant and torsion} 
\begin{prop} \label{prop:zdet} The first and second variations of the log of the determinant, 
$\log \det(\Delta_0 +s(s-1))$, are equal to the corresponding variations of $\log Z(s)$, that is, for $\Re(s) >1$,
$$\pa_\mu \log \det(\Delta_0 + s(s-1)) = \pa_\mu \log Z(s),$$
$$\bar\pa_\mu \pa_\mu \log \det(\Delta_0 + s(s-1)) = \bar\pa_\mu \pa_\mu \log Z(s).$$
For the case $s=1$, we have 
$$\pa_\mu \log \det(\Delta_0) = \pa_\mu \log Z'(1),$$
$$\bar\pa_\mu \pa_\mu \log \det(\Delta_0) = \bar\pa_\mu \pa_\mu \log Z'(1).$$
\end{prop} 
\begin{proof} The proof follows immediately from (\ref{sarnakzdet}).
\end{proof} 

As a corollary, we obtain the variation of the holomorphic analytic torsion.
\begin{cor} \label{cor:ator} 
The first and second variations of the logarithm of the holomorphic analytic torsion, $T_0 (X)$, defined in (\ref{torsions}), are respectively
$$\pa_\mu \log T_0(X)= \frac{1}{2} \pa_\mu \log Z'(1),$$
$$\bar\pa_\mu \pa_\mu \log \det(\Delta_0) = \frac{1}{2} \bar\pa_\mu \pa_\mu \log Z'(1).$$
\end{cor}

\subsubsection{Variation of the higher zeta functions} 
We conclude this section by computing the 
variation of the higher zeta functions.  
\begin{lemma}\label{higherzeta_variation}
For a fixed point $s \in \C$ with $\Re(s) > 1$, we have 
\begin{equation*}
\frac{\partial}{\partial \mu} \log z (s) = \sum_{\gamma \in \prim} \partial_\mu \log (l(\gamma)) \cdot \(   s \frac{d}{ds} \log z_\gamma(s)  -  \frac{d}{ds} \log \tilde{z}_\gamma (s)^{-1}    \),
\end{equation*}
where
\begin{equation*}
\tilde{z}_\gamma(s) = \prod_{m=1}^{\infty} (1 - e^{-l(\gamma) \cdot (s+m)})^{-m^2}.
\end{equation*}	
The sum in the right hand side of the equation above is absolutely convergent.
\end{lemma}

\begin{proof} 
First compute
\begin{equation*}
\begin{gathered}
\frac{\partial}{\partial \mu} (\log z_\gamma(s)) = \sum_{k \ge 0} \partial_\mu (- k \log(1- e^{-l(\gamma) \cdot (s+k)  })) = \\
\sum_{k \ge 0} -\partial_\mu(l(\gamma)) \cdot \frac{k(s+k) e^{-l(\gamma) \cdot (s+k)}}{1-e^{-l(\gamma)\cdot (s+k)}} = \sum_{k \ge 0} - \frac{\partial_\mu (l(\gamma)) \cdot k(s+k)}{e^{l(\gamma)\cdot (s+k)}-1} = \\ \sum_{k \ge 0} - \frac{\partial_\mu \(\log l(\gamma) \) \cdot l(\gamma) \cdot k(s+k)}{e^{l(\gamma)\cdot (s+k)}-1}.
\end{gathered}
\end{equation*}
Second,
\begin{equation*}
\begin{gathered}
\frac{d}{ds} \( \log z_\gamma (s)   \) = \frac{d}{ds} \( \sum_{k\ge 0}  -k \log\( 1 - e^{-l(\gamma) \cdot (s+k)}  \)   \) =  \sum_{k \ge 0}   \frac{- k \cdot l(\gamma)}{e^{l(\gamma)\cdot (s+k)}-1}.
\end{gathered}
\end{equation*}
Third,
\begin{equation*}
\begin{gathered}
\frac{d}{ds} (\log \tilde{z}_\gamma (s)^{-1}) = \frac{d}{ds} \( \sum_{k\ge 0}  k^2 \log\( 1 - e^{-l(\gamma) \cdot (s+k)}  \)   \) = \sum_{k \ge 0} \frac{k^2 \cdot l(\gamma)}{e^{l(\gamma)\cdot (s+k)}-1}.
\end{gathered}
\end{equation*}
To finish the lemma, we need to prove that we can change the order of differentiation with respect to $\mu$ (or $s$) and the summation over the set of primitive closed geodesics, that is 

\begin{equation}
\sum_{\gamma \in \prim} \partial_\mu (l(\gamma))   \sum_{k \ge 1} \frac{k^2 \cdot l(\gamma)}{e^{l(\gamma)\cdot (s+k)}-1} < \infty, \quad \sum_{\gamma \in \prim} \partial_\mu (l(\gamma))   \sum_{k \ge 1} \frac{k \cdot l(\gamma)}{e^{l(\gamma)\cdot (s+k)}-1} < \infty.
	\end{equation}
	\end{proof}

\begin{lemma}\label{hier_variation}
	For any fixed point $s \in \C$ with $\Re(s) > 1$ and $t > 1$, we have 
	\begin{equation*}
	\frac{\partial}{\partial \mu} \log z (s,t) = \sum_{\gamma \in \prim}  \( s \frac{d}{ds} \log z_\gamma (s,t) +  \sum_{j=1}^{t-1} \frac{d}{ds} \log z_\gamma (s,t-j) \),
	\end{equation*}
	The sum in the right hand side of the equation above is absolutely convergent in $\Re(s) > 1$.
\end{lemma}

\begin{proof}
	The proof follows the term-wise differentiation of $\frac{\partial}{\partial \mu} \log z (s,t)$. First we compute
	\begin{equation*}
	\begin{gathered}
	\frac{\partial}{\partial \mu} (\log z_\gamma(s)) = \sum_{k \ge 0} \partial_\mu \(\binom{t+k-1}{k} \log(1- e^{-l(\gamma) \cdot (s+k)  })\) = \\
	\sum_{k \ge 0} \partial_\mu(l(\gamma)) \cdot \frac{ \binom{t+k-1}{k} (s+k) e^{-l(\gamma) \cdot (s+k)}}{1-e^{-l(\gamma)\cdot (s+k)}} =  \sum_{k \ge 0}  \frac{\partial_\mu \(\log l(\gamma) \) \cdot l(\gamma) \cdot \binom{t+k-1}{k}(s+k)}{e^{l(\gamma)\cdot (s+k)}-1}.
	\end{gathered}
	\end{equation*}
    The differentiation is justified because $\binom{t+k-1}{k}(s+k)$ is a polynomial in $k$ and hence the right hand side of the the previous equation converges for $\Re(s) > 1$ by Lemma~\ref{convergence_lemma}.
	For $j \in \Z$, $0 \le  j < t$,
	\begin{equation*}
	\begin{gathered}
	\frac{d}{ds} (\log z_\gamma (s,t-j)) = \frac{d}{ds} \( \sum_{k\ge 0}  \binom{t+k-j-1}{k}  \log\( 1 - e^{-l(\gamma) \cdot (s+k)}  \)   \) = \\  \sum_{k \ge 0} \frac{ \binom{t+k-j-1}{k} \cdot l(\gamma)}{e^{l(\gamma)\cdot (s+k)}-1}, 
	\end{gathered}
	\end{equation*}
	and especially for $j = 0$,
	\begin{equation*}
	\begin{gathered}
	\frac{d}{ds} (\log z_\gamma (s,t)) =  \sum_{k \ge 0} \frac{ \binom{t+k-1}{k} \cdot l(\gamma)}{e^{l(\gamma)\cdot (s+k)}-1}.
	\end{gathered}
	\end{equation*}
As before, the differentiation is justified by Lemma \ref{convergence_lemma}. Note that 
		\begin{equation*}
		\binom{t+k-1}{k} (s+k) =  s \binom{t+k-1}{k} + \binom{t+k-1}{k+1}.     
		\end{equation*}
By the Christmas stocking identity, 
\begin{equation*}
 \binom{t+k-1}{k+1} = \sum_{j=1}^{t-1} \binom{t+k-j-1}{k},
\end{equation*}
\end{proof}

\begin{remark}
It is straightforward to repeat the calculations for the Selberg zeta function to compute the second variation of these zeta functions as well; this is left as an exercise for the reader.
\end{remark} 

\section{Asymptotic behavior as $\Re s \to \infty$}
We begin by estimating the terms $A_\gamma(s)$ and $B_\gamma(s)$ when $\Re s \to \infty$.  

\begin{prop}  \label{propagamma} 
We have for each $\gamma$ 
$$\lim_{\Re s \to \infty} \frac{A_\gamma(s) e^{sl(\gamma)} (1-e^{-l(\gamma)})}{s l(\gamma)} = 1.$$
\end{prop}

\begin{proof} 
We write 
$$A_\gamma(s) = \frac{ s l(\gamma)}{e^{s l(\gamma)}} \sum_{k \geq 0} \frac{ 1}{e^{k l(\gamma)} - e^{- s l(\gamma)}} + \frac{l(\gamma)}{e^{s l(\gamma)}} \sum_{k \geq 0} \frac{k}{e^{k l(\gamma)} - e^{- s l(\gamma)}}.$$
Hence, we define 
$$I_a (s) := \sum_{k \geq 0} \frac{ 1}{e^{k l(\gamma)} - e^{- s l(\gamma)}}, \quad II_a (s) =  \sum_{k \geq 0} \frac{k}{e^{k l(\gamma)} - e^{- s l(\gamma)}}.$$
Since we are interested in the behavior as $\Re s \to \infty$, we may assume that 
$$\Re s > \frac{\ln(2)}{l_0}.$$
Then we have 
$$0< | e^{-s l(\gamma)}|  < \frac{1}{2} \leq \frac{e^{k l(\gamma)}}{2} \quad \forall k \geq 0  \implies  \left|  \frac{ 1}{e^{k l(\gamma)} - e^{- s l(\gamma)}} \right| < 2 e^{- k l(\gamma)},$$
hence 
$$0< | I_a(s)|  \leq \sum_{k \geq 0} 2e^{-k l(\gamma)} = \frac{2}{1-e^{-l(\gamma)}}.$$
More generally, by the absolute convergence of the series, the function 
$$f(z) = \sum_{k \geq 0} \frac{1}{e^{k l(\gamma)} - z},$$
is continuous for $|z|<1$.  In particular, it is continuous at $z=0$.  Hence, using 
$$e^{-s l(\gamma)} = z,$$
we have
$$\lim_{z \to 0} f(z) = \lim_{\Re s \to \infty} I_a (s) = f(0) = \frac{1}{1-e^{-l(\gamma)}}.$$

Next, we estimate the $II_a$ term, 
$$0 < |II_a (s)| < 2 \sum_{k \geq 0} k e^{-k l(\gamma)} = \left . - 2 \frac{d}{dx} \sum_{k \geq 0} e^{-k x} \right|_{x=l(\gamma)} = \frac{2 e^{- l(\gamma)}}{(1-e^{- l(\gamma)})^2}.$$
Consequently, 
$$\lim_{\Re s \to \infty} \frac{A_\gamma(s) e^{sl(\gamma)} (1-e^{-l(\gamma)})}{s l(\gamma)} = \lim_{\Re s \to \infty} I_a (s) (1-e^{-l(\gamma)}) + \frac{II_a(s) (1 - e^{-l(\gamma)})}{s} = 1.$$
\end{proof} 

In a similar way, we compute the asymptotic behavior of $B_\gamma(s)$ as $\Re s \to \infty$.  

\begin{prop} \label{propbgamma} 
We have for each $\gamma$ 
$$\lim_{\Re s \to \infty} - \frac{B_\gamma(s) e^{s l(\gamma)} (1-e^{-l(\gamma)})}{l(\gamma)^2 s^2} = 1.$$
\end{prop}

\begin{proof}
The proof is quite similar to that of the preceding proposition.  We write 
$$B_\gamma (s) = - l(\gamma)^2 \left[ s^2 \sum_{k \geq 0} \frac{e^{(s+k) l(\gamma)}}{(e^{(s+k) l(\gamma)} -1)^2} + 2s \sum_{k \geq 0} \frac{k e^{(s+k) l(\gamma)}}{(e^{(s+k) l(\gamma)} -1)^2} + \sum_{k \geq 0} \frac{k^2 e^{(s+k) l(\gamma)}}{(e^{(s+k) l(\gamma)} -1)^2} \right].$$
We therefore define three terms, 
$$I_b (s) := \sum_{k \geq 0} \frac{e^{(s+k) l(\gamma)}}{(e^{(s+k) l(\gamma)} -1)^2}, \quad II_b(s) = \sum_{k \geq 0} \frac{k e^{(s+k) l(\gamma)}}{(e^{(s+k) l(\gamma)} -1)^2},$$
and 
$$III_b (s) =  \sum_{k \geq 0} \frac{k^2 e^{(s+k) l(\gamma)}}{(e^{(s+k) l(\gamma)} -1)^2}.$$

We begin by computing that 
$$I_b (s) = \sum_{k \geq 0} \frac{1}{(e^{(s+k) l(\gamma)} -1)(1-e^{-(s+k) l(\gamma)})} = \frac{1}{e^{sl(\gamma)}} \sum_{k \geq 0} \frac{1}{(e^{k l(\gamma)} - e^{-sl(\gamma)})( 1-e^{-(s+k) l(\gamma)})}.$$
We note that 
$$\lim_{s \to \infty} 1 - e^{-(s+k) l(\gamma)} = 1 \quad \forall k \geq 0,$$
and that 
$$|1 - e^{-(s+k) l(\gamma)}| \geq \frac 1 2, \quad \textrm{ for all } k \geq 0, \textrm{ when } \Re s > \frac{\ln(2)}{l_0}.$$
The function 
$$f(z) = \sum_{k \geq 0} \frac{1}{(e^{k l(\gamma)} - z)( 1- z e^{-k l(\gamma)})}$$
is continuous for all $|z| < 1$.  In particular, this function is continuous at $z=0$.  We therefore have, with 
$$z = e^{-s l(\gamma)},$$
$$\lim_{z \to 0} f(z) = f(0) = \sum_{k \geq 0} e^{-k l(\gamma)} = \frac{1}{1-e^{-l(\gamma)}}.$$
Hence, we see that 
\begin{equation} \label{ibs} \lim_{\Re s \to \infty} e^{s l(\gamma)} I_b(s) (1-e^{-l(\gamma)}) = 1. \end{equation} 
We shall estimate the other two terms, 
$$0 < | II_b (s) | = \left| \frac{1}{e^{s l(\gamma)}} \sum_{k \geq 0} \frac{k}{(e^{k l(\gamma)} - e^{-sl(\gamma)})( 1-e^{-(s+k) l(\gamma)})} \right| \leq \left| \frac{1}{e^{s l(\gamma)}} \right| \sum_{k \geq 0} \left| \frac{2k}{(e^{k l(\gamma)} - e^{-sl(\gamma)})} \right| $$
$$ \leq \left| \frac{4}{e^{s l(\gamma)}} \right| \sum_{k \geq 0} k e^{-k l(\gamma)} = \frac{4 e^{-l(\gamma)}}{|e^{s l(\gamma)}| (1-e^{-l(\gamma)})^2}. $$
We therefore see that 
\begin{equation} \label{iibs} \lim_{\Re s \to \infty} \frac{II_b (s) e^{s l(\gamma)}}{s} = 0. \end{equation} 

Next, 
$$0 <| III_b (s) | = \frac{1}{|e^{s l(\gamma)}|} \left| \sum_{k \geq 0} \frac{k^2}{(e^{k l(\gamma)} - e^{-sl(\gamma)})( 1-e^{-(s+k) l(\gamma)})}  \right| \leq \frac{4}{|e^{s l(\gamma)}|} \sum_{k \geq 0} k^2 e^{-k l(\gamma)}.$$
We compute the sum 
$$\sum_{k \geq 0} k^2 e^{-kx} = \frac{d^2}{dx^2} \sum_{k \geq 0} e^{-kx} = \frac{e^{-x}}{(1-e^{-x})^2} + \frac{2 e^{-2x}}{(1-e^{-x})^3}.$$
Hence, 
$$0 < | III_b (s) | \leq \frac{4}{|e^{s l(\gamma)}|} \left( \frac{e^{-l(\gamma)}}{(1-e^{-l(\gamma)})^2} + \frac{2 e^{-2 l(\gamma)}}{(1-e^{-l(\gamma)})^3} \right). $$
In particular, 
\begin{equation} \label{iiibs} \lim_{\Re s \to \infty} \frac{III_b (s) e^{s l(\gamma)}}{s^2} = 0. \end{equation} 

Recalling that 
$$B_\gamma(s) = - l(\gamma)^2 \left( s^2 I_b (s) + 2s II_b (s) + III_b (s) \right),$$
by \eqref{ibs}, \eqref{iibs}, and \eqref{iiibs} we see that 
$$\lim_{\Re s \to \infty}  - \frac{B_\gamma(s) e^{s l(\gamma)} (1-e^{-l(\gamma)})}{l(\gamma)^2 s^2} $$
$$= \lim_{\Re s \to \infty} I_b(s) e^{s l(\gamma)} (1-e^{-l(\gamma)}) + \frac{2 II_b (s) e^{s l(\gamma)} (1-e^{-l(\gamma)})}{s} + \frac{III_b(s) e^{s l(\gamma)} (1-e^{-l(\gamma)})}{s^2} = 1.$$ 
\end{proof}

We are now poised to complete the proof of Theorem \ref{thmasy}.  
\begin{proof}[Proof of Theorem \ref{thmasy}]  
We first assume, recalling \eqref{abusenot}, 
$$|\pa_\mu  l(\gamma_0)|^2 \neq 0.$$
This immediately implies 
$$\pa_\mu \log l(\gamma_0) \neq 0 \implies |\pa_\mu \log l(\gamma_0)|^2 \neq 0.$$
For all $\gamma$ with $l(\gamma) > l_0$, we note that by Propositions \ref{propagamma} and \ref{propbgamma}, 
$$\lim_{\Re s \to \infty} \frac{ A_\gamma(s)}{B_{\gamma_0} (s)} =0 = \lim_{\Re s \to \infty} \frac{ B_\gamma(s)}{B_{\gamma_0} (s)}.$$
Hence we see that by Theorem 1, 
$$\lim_{\Re s \to \infty} \frac{ \bar \pa_\mu \pa_\mu \log Z(s)}{B_{\gamma_0}(s)} = 1.$$
This together with Proposition \ref{propbgamma} for the asymptotics of $B_{\gamma_0} (s)$ as $\Re s \to \infty$ completes the proof of the theorem in this case.  

Next, we assume that 
$$\pa_\mu  l_0 = 0.$$
This shows that the coefficient of $B_{\gamma_0}(s)$ vanishes.  By \cite[Corollary 1.3]{Ax-Sch-2012}
$$\bar \pa_\mu \pa_\mu l(\gamma) \neq 0, \quad \forall \gamma \in \Gamma.$$
Hence, the term with the slowest exponential decay as $\Re s \to \infty$ is $A_{\gamma_0} (s)$.  We similarly see that for all $\gamma$ with $l(\gamma) > l_0 = l(\gamma_0)$ that 
$$\lim_{\Re s \to \infty} \frac{B_\gamma(s)}{A_{\gamma_0} (s)} = 0 = \lim_{\Re s \to \infty} \frac{A_{\gamma} (s)}{A_{\gamma_0} (s)}.$$
We therefore have by Theorem 1, 
$$\lim_{\Re s \to \infty} \frac{ \bar \pa_\mu \pa_\mu \log Z(s)}{A_{\gamma_0} (s)} = 1.$$
The proof is then completed by the asymptotics of $A_{\gamma_0}(s)$ as $s \to \infty$ given in Proposition \ref{propagamma}. The statements for the Ruelle zeta function follow immediately upon noting that $\log R(s)$ is simply given by the $k=0$ term in $\log Z(s)$.

For the Hilbert-Schmidt norm of the squared resolvent, 
$$\bar \pa_\mu \pa_\mu  || (\Delta_0 + s(s-1))^{-1} ||_{HS} ^2 = - L^2 \bar \pa_\mu \pa_\mu \log Z(s).$$
Recalling the definition of $L$ from Lemma \ref{le:hsnorm} we compute that 
$$L^2 = -  \frac{2}{(2s-1)^3} \frac {d}{ds} + \frac{1}{(2s-1)^2} \frac{d^2}{ds^2}.$$
In the case $|\pa_\mu  l(\gamma_0)|^2 \neq 0$, the asymptotic behavior of $\bar \pa_\mu \pa_\mu  || (\Delta_0 + s(s-1))^{-1} ||_{HS} ^2$ as $\Re s \to \infty$ is therefore given by the dominant asymptotic behavior of 
$$\frac{ |\pa_\mu l_0|^2}{1-e^{-l_0}} \left(  -  \frac{2}{(2s-1)^3} \frac {d}{ds} + \frac{1}{(2s-1)^2} \frac{d^2}{ds^2} \right) s^2 e^{-s l_0}$$
$$ = \frac{|\pa_\mu  l_0|^2}{(1-e^{-l_0})(2s-1)^2} \left[  - \frac{2 \left( 2se^{-sl_0} - s^2 l_0 e^{-sl_0}\right)}{2s-1} + 2e^{-sl_0} - 2sl_0 e^{-sl_0} - 2s l_0 e^{-s l_0} + s^2 l_0^2 e^{-s l_0} \right]$$
$$= \frac{|\pa_\mu  l_0|^2 l_0 ^2 e^{-s l_0} }{4(1-e^{-l_0})} \left( 1 + \mathcal O(s^{-1}) \right), \quad \Re s \to \infty.$$ 

We therefore have in this case
$$\lim_{\Re s \to \infty} \frac{\bar \pa_\mu \pa_\mu  || (\Delta_0 + s(s-1))^{-1} ||_{HS} ^2 4 (1-e^{-l_0})}{|\pa_\mu l_0|^2 l_0^2 e^{-s l_0}} = 1.$$

In the case $\pa_\mu  l_0 = 0$, the asymptotic behavior of $\bar \pa_\mu \pa_\mu  || (\Delta_0 + s(s-1))^{-1} ||_{HS} ^2$ as $\Re s \to \infty$ is given by that of 
$$-L^2 \bar \pa_\mu \pa_\mu \log l_0 \left( \frac{s l_0 e^{-s l_0}}{1-e^{-l_0}} \right)$$
$$=\frac{\bar \pa_\mu \pa_\mu \log l_0}{1-e^{-l_0}} \left(\frac{2}{(2s-1)^3} \left( e^{-sl_0} - l_0 s e^{-s l_0} \right) - \frac{1}{(2s-1)^2} \left( - l_0 e^{-s l_0} - l_0 e^{-s l_0} + l_0^2 s e^{-s l_0} \right) \right)$$
$$= - \frac{l_0^2 e^{-s l_0} \bar \pa_\mu \pa_\mu \log l_0}{4s(1-e^{-l_0})} \left( 1+ O(1/s) \right) \quad \Re s \to \infty.$$  
We therefore have 
$$\lim_{\Re s \to \infty} - \frac{\bar \pa_\mu \pa_\mu  || (\Delta_0 + s(s-1))^{-1} ||_{HS} ^2 4 s (1-e^{-l_0})}{l_0 ^2 e^{-s l_0} \bar \pa_\mu \pa_\mu \log l_0} = 1.$$

\end{proof} 

\subsection{Curvature asymptotics}  
We now consider the special case when $s=m \in \N$, beginning by recalling 

\begin{prop}\label{TZ-form} [Theorem 2 \cite{ZT1987}] 
The second variation 
\begin{equation} \label{2ndvar1} \bar\partial_\mu\partial_\mu \log Z'(1) =-\cher^{(1)}(\mu, \mu) + \frac{1}{12\pi} \Vert \mu\Vert^2_{WP}. \end{equation} 
Let $m \geq 2$.  Then, the second variation 
\begin{equation} \label{2ndvarm} \bar\partial_\mu\partial_\mu \log Z(m) = -\cher^{(m)}(\mu, \mu)  + \frac{6(m-1)^2+6(m-1) +1}{12\pi} \Vert \mu\Vert^2_{WP}.
\end{equation} 
\end{prop}

The curvature  $\cher^{(m)}(\mu, \mu)$ has been known to admit an expansion of the form $c_2 m^2 +c_1 m + c_0 + \cdots$  for large $m$.\footnote{We began by computing the coefficients of $m^{-k}$ for $k\in \N$ using the expression of the curvature given in Proposition \ref{Bob-prop}.  These coefficients become increasingly complicated expressions at an exponential rate, and in the end, they vanished in each increasingly lengthy calculation.  Indeed, our results show that there are no further non-zero coefficients $c_k m^k$ for $k \in \Z$ with $k \leq -1$, because the remainder is $O(k^{-N})$ for any $N \in \N$ as $k \to \infty$.}  In \cite{Ma-Zh} the first two coefficients have been found explicitly for general fibration of K\"a{}hler manifolds. We determine $c_0$ and the remainder in Corollary 2.  

\begin{proof}[Proof of Corollary 2]
We use \eqref{2ndvarm} to write  
$$\bar\partial_\mu\partial_\mu \log Z(m) = -\cher^{(m)}(\mu, \mu)  + \frac{6m(m-1)+1}{12\pi} \Vert \mu\Vert^2_{WP}.$$
We have proven that 
$$\lim_{m \to \infty} - \frac{\bar\partial_\mu\partial_\mu \log Z(m) e^{m l_0} (1-e^{-l_0})}{l_0^2 m^2 |\pa_\mu \log l_0|^2} = 1.$$
This shows that 
$$\lim_{m \to \infty} \left( -\cher^{(m)}(\mu, \mu)  + \frac{6m(m-1)+1}{12\pi} \Vert \mu\Vert^2_{WP} \right) \frac{e^{m l_0}}{m^2} = \frac{l_0^2 |\pa_\mu \log l_0|^2}{(1-e^{-l_0})}.$$
Hence, 
$$-\cher^{(m)} (\mu, \mu) +  \frac{6m(m-1)+1}{12\pi} \Vert \mu\Vert^2_{WP} = O(m^2 e^{-m l_0}), \quad m \to \infty,$$
which shows that 
$$\cher^{(m)} (\mu, \mu) = \frac{6m(m-1)+1}{12\pi} \Vert \mu\Vert^2_{WP} + O(m^2 e^{-m l_0}), \quad m \to \infty.$$
\end{proof}

\begin{remark} In a recent preprint \cite{Wan-Zhang} the third author together with Xueyuan Wan  been able to prove a general result
in the setting of  families of K\"a{}hler manifolds.  The result shows that the leading three terms of 
the curvature $\cher^{(m)}(\mu, \mu)$ and the Quillen curvature agree. In the case of Riemann surfaces this can be proved independently using the Bergman kernel expansion  \cite{Lu-AJM} and Berndtsson's curvature formula in Proposition \ref{Bob-prop}.
\end{remark}

\subsection{The cases $m=1, 2$.}
Here we  consider the special cases $m=1, 2$.  For $m=1$ we show that there are surfaces and harmonic Beltrami differentials for which the second variation is strictly positive, as well as surfaces and harmonic Beltrami differentials for which the second variation is strictly negative.  In \cite[Lemma 4.6]{wolpert86}, Wolpert obtained an upper estimate of $-\frac{1}{2\pi(g-1)}$ for the curvature  $-\cher^{(2)}$. We shall use Wolf's pointwise estimates \cite{Wolf-JDG-2012} for solutions of the Laplace equation to find a sharper estimate, $-\frac{1}{6\pi(g-1)}$. This might also be known to experts. Consequently we find also estimates for the variation of the Selberg zeta function.

\begin{prop} Let $m=1$. \begin{enumerate}
\item Let the genus $g\ge 3$. Then there exists $t \in \mathcal T_g$ such that the corresponding Riemannian surface, $X=X_t$ is hyperelliptic.  Moreover, there exist harmonic Beltrami differentials, $\mu$ such that 
$ \bar\partial_\mu\partial_\mu \log Z'(1) >0$.
\item Let $g \ge 2$. At any point $t \in \mathcal T_g$ there exists a Beltrami differential $\mu$ such that for the corresponding Riemannian surface, $X=X_t$, $ \bar\partial_\mu\partial_\mu \log Z'(1) <0.$ 
\end{enumerate}
\end{prop} 

\begin{proof} If  $X$ is hyperelliptic  then there exist harmonic Beltrami differentials $\mu$ such that 
$$ \langle R(\mu, \mu)u, u\rangle =0 $$
for all $u\in H^{0}(\mathcal K)$.  This follows for example from \cite[Lemma 3.3 \& Proposition 3.4]{Bo-mz}. 
Thus by Propositions \ref{TZ-form} and \ref{Bob-prop}, we have 
$$ \bar\partial_\mu\partial_\mu  \log Z'(1) > 0.$$

To prove the second statement, for any Riemann surface, $X$, let $\omega$ be any non-zero abelian differential
of norm 1.  Then $\omega^2$ is a holomorphic quadratic differential. Let $\mu=\rho^{-1} {\overline \omega^2}$ be the corresponding harmonic Beltrami differential.  We compute the curvature $\cher^{(1)}(\mu, \mu)$ using Proposition \ref{Bob-prop}.  To do this, let $\{\omega_j\} $ be an orthonormal basis of $H^0(\mathcal K)$, and fix 
$\omega_1=\omega$.  The sum $\Vert [\mu \cdot \omega_i]\Vert^2$
is
$$\sum_i\Vert [\mu \cdot \omega_i]\Vert^2 \ge \langle\mu \cdot\omega, \bar{\omega}\rangle^2 $$

and
$$\langle \mu \cdot\omega, \bar{\omega}\rangle=\int_{X} \rho^{-1} \overline{\omega}^2 \omega {\omega}  =\Vert{\omega^2}\Vert^2 =\Vert{\mu}\Vert^2_{WP}.$$
Thus 
$$-\cher^{(1)} (\mu, \mu) \le -\Vert{\mu}\Vert^2_{WP},$$
and 
$$\bar\partial_\mu\partial_\mu  \log Z'(1) =-\cher^{(1)}(\mu, \mu) +\frac{1}{12\pi}\Vert \mu\Vert^2 _{WP} \le 
\(-1+\frac{1}{12\pi}\)\Vert \mu\Vert^2 _{WP} <0.$$
\end{proof}

The case  $m=2$ is of special interest because $\mathcal H^0(\mathcal K^2)$ can be viewed as the dual of the tangent space of Teichm\"u{}ller space.  Moreover, the negative of the curvature, $- \cher^{(2)}$ is therefore the Ricci curvature of the Weil-Petersson metric.  This object has been studied intensively for quite some time; see for example \cite{wolpert86}. We demonstrate elementary upper and lower estimates of the variation in this case.  

\begin{prop} Let $m=2$ and $B=B^{(2)}$ be the Bergman kernel in this case. We have     
$$\bar\partial_\mu\partial_\mu  \log Z(2) \ge  \frac{13}{12\pi} \Vert\mu \Vert^2 _{WP} -  \Vert \mu\Vert_{L^4 (X)} ^2 \Vert B \Vert_{L^2 (X) }  - \Vert |\mu|^2 B\Vert_{L^1(X)}$$
and 
$$\bar\partial_\mu\partial_\mu  \log Z(2) \le \left(\frac{13}{12\pi} - \frac{1}{6\pi(g-1)} \right)  \Vert\mu \Vert^2_{WP}.$$
\end{prop}

\begin{proof}
In this case we have 
$$\bar\partial_\mu\partial_\mu  \log Z(2) = - \cher^{(2)} (\mu, \mu) + \frac{12}{12\pi} \Vert \mu \Vert^2 _{WP}.$$
We shall prove the estimates by demonstrating upper and lower estimates for $\cher^{(2)} (\mu, \mu)$ in this case.  In this way, we are also simultaneously demonstrating lower and upper estimates for the Ricci curvature of the Weil-Petersson metric.  

By Proposition \ref{Bob-prop} 
$$\cher^{(2)} (\mu, \mu) = \int_X f(\mu) B + \sum_{j=1} ^{d_2} \langle
(\Box_{1,1} +1)^{-1} (\mu \cdot u_j), \mu \cdot u_j \rangle,$$
with
$$f({\mu})=(1+\square_0)^{-1}|\mu|^2.$$
By the Cauchy-Schwarz inequality 
$$\left| \int_X f(\mu) B \right| \leq ||f(\mu)||_{L^2 (X)} ||B||_{L^2 (X)}.$$
We have the operator estimates $(\Box_{1,1} + 1)^{-1} \leq 1$, and $(\Box_0
+1)^{-1} \leq 1$.  In this case, we actually have the
unitary equivalence of $\Box_{1,1}$ and $\Box_0$ via the natural
identification
of $f(z) dz\otimes d\bar z$ with $f(z)y^2$, as
the metric tensor $y^{-2}dz\otimes d\bar z$ is globally defined.  Consequently 
$$||f(\mu)||_{L^2 (X)} \leq ||\mu^2||_{L^2 (X)} = ||\mu||_{L^4 (X)} ^2.$$
Hence,  
$$\cher^{(2)} (\mu, \mu) \leq \Vert \mu \Vert_{L^4(X)} ^2 \Vert B \Vert_{L^2(X)} +  \sum_{j=1} ^{d_2} \langle \mu \cdot u_j, \mu \cdot u_j \rangle.$$
As for the second term, this is just 
$$\sum_{j=1} ^{d_2} \langle \mu \cdot u_j, \mu \cdot u_j \rangle = \int_X |\mu|^2 B,$$
so we have the lower bound 
$$\bar\partial_\mu\partial_\mu  \log Z(2) \geq \frac{13}{12\pi} \Vert \mu \Vert^2 _{WP} - \Vert \mu \Vert_{L^4(X)} ^2 \Vert B \Vert_{L^2(X)} - \Vert |\mu|^2 B \Vert_{L^1(X)}.$$ 
This implies the lower estimate.  

To prove the upper estimate, we select one of the basis elements, $u_j$, to be the corresponding quadratic form of $\mu$.  So, for example, we set  
$$u_1 = \frac 1{\Vert \mu\Vert _{WP}} \rho \overline{\mu},$$
We therefore have 
$$\cher^{(2)} (\mu, \mu) \geq  \int_X (\Box_0 +1)^{-1} (|\mu|^2) B + \langle (\Box_{1,1} +1)^{-1} (\mu \cdot u_1), \mu \cdot u_1 \rangle$$
$$\geq \int_X (\Box_0 +1)^{-1} (|\mu|^2) |u_1|^2 +  \langle (\Box_{1,1} +1)^{-1} (\mu \cdot u_1), \mu \cdot u_1 \rangle.$$
It follows from the unitary equivalence of the operators $\Box_{1,1}$ and $\Box_0$  that 
\begin{equation} \label{ricm2} \cher^{(2)}(\mu, \mu) \ge   \frac 2{\Vert \mu\Vert ^{2} _{WP} } \langle  (1+\square_0)^{-1}|\mu|^2, |\mu|^2 \rangle  \end{equation} 

We proceed to estimate using \cite{Wolf-JDG-2012}.  On \cite[p. 30]{Wolf-JDG-2012}, we see that in Wolf's notation, our $\mu = \frac{\Phi}{g_0}$.  We also note that the operator denote $\Delta$ there is equal to $-\Delta_0 = -2\Box_0$ in our notation.  Hence, the statement of \cite[Lemma 5.1]{Wolf-JDG-2012} is in our context 
$$0 \leq v \leq -2(-2\Box_0 -2)^{-1} |\mu|^2, \quad v = \frac{|\mu|^2}{3}.$$
Of course, this is equivalent to 
$$0 \leq \frac{|\mu|^2}{3} \leq (\Box_0 +1)^{-1} |\mu|^2.$$
We use this estimate in (\ref{ricm2}) to obtain 
\begin{equation} \label{ricm2b} \cher^{(2)}(\mu, \mu) \geq \frac 2 3 \frac{1}{\Vert \mu\Vert^2_{WP} } \int_X |\mu|^4 . \end{equation} 
The Cauchy-Schwarz inequality gives 
$$\sqrt{ \int_X |\mu|^4 } \sqrt{ \int_X 1^2 } \geq \int_X |\mu|^2 \cdot 1 \implies \int_X |\mu|^4 \geq \( \int_X |\mu|^2 \)^2 \( \int_X 1 \)^{-1}.$$
So, recalling the fact that the volume of our surface is $4\pi(g-1)$, we have the lower estimate 
$$\int_X |\mu|^4 \geq \Vert \mu \Vert_{WP} ^4 \frac{1}{4\pi (g-1)}.$$ 
Consequently, putting this estimate into \eqref{ricm2b}, 
$$\cher^{(2)} (\mu, \mu) \geq  \frac{ \Vert \mu \Vert ^2 _{WP}}{6\pi(g-1)}.$$
\end{proof}

It would be interesting to find some more effective estimates for the variations in terms of the eigenvalues of $\square_{m-1, 1}$ and the geometry of the Riemann surface.  The related questions of estimating  the Weil-Petersson sectional curvature has been studied extensively; see \cite{Wolpert-JDG} and references therein.

\begin{appendix} 
\section{The Hilbert-Schmidt norm of the squared resolvent
and its variation}  \label{appendixa} 

\begin{proof}[Proof of Lemma \ref{le:hsnorm}] 
Let 
$$N(\gamma) = e^{\ell(\gamma)}, \quad \alpha = s-1/2, \quad \beta = s'-1/2, \quad c(r) = r \tanh(\pi r).$$
For non-prime $\gamma$, with 
$$\gamma = k \gamma_p, \quad \gamma_p \in \prim,  \quad N_p(\gamma) := N(\gamma_p).$$
Recall that by the Gauss-Bonnet Theorem, the area of our Riemann surface is $4\pi(g-1)$.  By the Selberg Trace Formula, 
$$R(s, s'):=\sum_{n \geq 0} \frac{1}{r_n^2 + \alpha^2} - \frac{1}{r_n^2 + \beta^2} =$$
$$ (g-1) \int_{\R} \left( \frac{1}{r^2 + \alpha^2} - \frac{1}{r^2 + \beta^2} \right) c(r) dr + \frac{1}{2\alpha} \sum_{\gamma \in \cL} \frac{\log N_p(\gamma) N(\gamma)^{-s} }{1-N(\gamma)^{-s}} - \frac{1}{2\beta} \sum_{\gamma \in \cL} \frac{ \log N_p(\gamma) N(\gamma)^{-s'}}{1-N(\gamma)^{-s'}}.$$

Here note that we are taking $\gamma$ not necessarily primitive, that is $\gamma \in \cL$ rather than $\cL_p$.  Hence, $N(\gamma) = e^{\ell(k \gamma)} = e^{k \ell(\gamma)}$.  The eigenvalues, 
$$r_n ^2 + \alpha^2 = r_n^2 + 1/4 + s(s-1), \quad r_n^2 + 1/4 = \lambda_n, \quad \sigma(\Delta_0) = \{ \lambda_n\}_{n \in \N}.$$
Thus, the spectrum of the operator 
$$\sigma( \Delta_0 + s(s-1)) = \{ \lambda_n + s(s-1)\}_{n \in \N} = \{ r_n^2 + \alpha^2 \}_{n \in \N}.$$
Then, the spectrum of the inverse, 
$$\sigma( (\Delta_0 + s(s-1)^{-1}) = \left\{ \frac{1}{r_n^2 + \alpha^2} \right\}.$$
Hence, the square of the Hilbert-Schmidt norm 
$$|| (\Delta_0  + s(s-1))^{-1} ||_{HS} ^2 = \sum_{n \geq 0 } \frac{1}{(r_n^2 + \alpha^2)^2}.$$
We divide the expression for $R(s, s')$ by $s-s'$ and let $s' \to s$, thereby obtaining the derivative with respect to $s$.  Thus, we obtain 
$$\lim_{s' \to s} \frac{R(s, s')}{s-s'} =- 2\alpha \sum_{n \geq 0} \frac{1}{(r_n^2 + \alpha^2)^2} = - 2 \alpha || (\Delta_0 + s(s-1))^{-1} || _{HS} ^2.$$
By the Selberg Trace Formula, this is equal to 
$$(g-1) \int_\R \frac{-2\alpha r \tanh(\pi r) }{(r^2 + \alpha^2)^2} dr - \frac{1}{2\alpha^2} \sum_{\gamma} \frac{\log(N(\gamma)) N(\gamma)^{-s} }{1-N(\gamma)^{-s}} - \frac{1}{2\alpha} \sum_{\gamma} \frac{ \log(N(\gamma)) ^2 N(\gamma)^s}{(N(\gamma)^s - 1)^2}.$$

Note that in terms of $\alpha$, 
$$\frac{1}{2\alpha} \frac{d}{ds} = L.$$
Moreover, 
$$\log Z(s) = \sum_{\gamma \in \cL} \log(1-N(\gamma)^{-s}).$$
Hence, 
$$L \log Z(s) = \frac{1}{2\alpha} \sum_{\gamma \in \cL} \frac{\log N(\gamma) N(\gamma)^{-s}}{1-N(\gamma)^{-s}}.$$
Thus, 
$$- 2 \alpha || (\Delta + s(s-1))^{-1} || _{HS} ^2 = (g-1) \int_\R \frac{-2\alpha r \tanh(\pi r) }{(r^2 + \alpha^2)^2} dr - \frac{d}{ds}L(\log Z(s)).$$
Dividing by $-1/2\alpha$ we then have 
$$|| (\Delta + s(s-1))^{-1} || _{HS} ^2 = \frac{-(g-1)}{2\alpha}   \int_\R \frac{-2\alpha r \tanh(\pi r) }{(r^2 + \alpha^2)^2} dr + L^2 (\log Z(s))$$
$$=(g-1) \int_\R \frac{r \tanh(\pi r) }{(r^2 + \alpha^2)^2} dr + L^2 (\log Z(s)).$$

We consider the integral.  Let 
$$f(z) = \frac{z \tanh (\pi z)}{(z^2 + \alpha^2)^2}.$$
If $|z| \to \infty$, $|f(z)| = O(|z|^{-3})$, thus we can compute the integral over $\R$ using a large half disk contour together with the residue theorem.  The estimate shows that the integral on the curved arc of the half disk is vanishing as the radius of the disk tends to infinity.  Hence 
$$\int_R f(z) dz = 2\pi i \sum_{ \Im(z)>0} Res(f; z).$$
We compute the residues.  The hyperbolic tangent has simple poles wherever the hyperbolic cosine vanishes.  For $\Im(z) > 0$, this occurs at 
$$z = i(n+1/2), n \in \Z, n \geq 0.$$
Moreover, there is also a pole of order two at $z = i \alpha$.  The residues at the simple poles are 
$$\lim_{z \to i (n+1/2)} \frac{z-i(n+1/2)}{\cosh(\pi z)} \frac{z \sinh (\pi z)}{z^2 + \alpha^2)^2}.$$
The first part converges to 
$$\frac{1}{\cosh(\pi z)'}, \quad z = i(n+1/2) = \frac{1}{\pi \sinh(\pi i (n+1/2))} = \frac{ (-1)^{n}}{\pi}.$$
The second part converges to 
$$\frac{i (n+1/2) (-1)^{n} }{ (\alpha^2 - (n+1/2)^2)^2}.$$
Thus the residues at these poles are 
$$\frac{i(n+1/2)}{\pi (\alpha^2 - (n+1/2)^2)^2}.$$

The residue at the pole of order 2 is $F'(i \alpha)$, where $F(z) := (z- i\alpha)^2 f(z)$.  Hence, we compute the derivative of 
$$F(z) = \frac{z \tanh (\pi z)}{(z+i\alpha)^2}, \quad F'(z) = \frac{ (z+i\alpha)^2 (\tanh (\pi z) + \pi z (1-\tanh^2 (\pi z))) - 2(z+i\alpha) z \tanh(\pi z)}{(z+i\alpha)^4}.$$
Hence 
$$F'(i\alpha) = \frac{ \tanh(i\pi \alpha) + i \pi \alpha(1-\tanh^2 (i \pi \alpha))}{(2i\alpha)^2} - \frac{2i\alpha \tanh(i \pi \alpha)}{(2i\alpha)^3} = \frac{\pi}{2i\alpha} (1-\tanh^2(i\pi \alpha)).$$
We have by the Residue Theorem, 
$$\int_\R \frac{r \tanh(\pi r)}{(r^2 + \alpha^2)^2} dr = (2\pi i) \left( \frac{\pi}{2i\alpha}(1-\tanh^2(i\pi \alpha)) + \sum_{n \geq 0} \frac{i (n+1/2)}{\pi (\alpha^2 - (n+1/2)^2)^2}\right)$$
$$= \frac{\pi^2}{\alpha} (1+ \tan^2(\pi \alpha)) - 2 \sum_{n \geq 0} \frac{n+1/2}{(\alpha^2 - (n+1/2)^2)^2}.$$
This is due to the fact that 
$$\tanh(i\pi \alpha) = i \tan(\pi \alpha), \quad \alpha \in \R.$$
To achieve a bit more simplification, we use the fact that in general 
$$\frac{1}{(\alpha+\beta)^2 (\alpha - \beta)^2} = \left( \frac{1}{(\alpha+ \beta)^2} - \frac{1}{(\alpha - \beta)^2} \right) \frac{-1}{4\alpha \beta}.$$ 

Then we have for the sum, 
$$\sum_{n \geq 0} \frac{ (n+1/2)}{ (\alpha + (n+1/2))^2 (\alpha - (n+1/2))^2} $$
$$= \sum_{n \geq 0} \frac{-(n+1/2)}{4\alpha (n+1/2)} \left( \frac{1}{(\alpha+ n+1/2)^2} - \frac{1}{(\alpha - (n+1/2))^2} \right) $$
$$= -\frac{1}{\alpha} \sum_{n \geq 0} \left( \frac{1}{ (2(s-1/2) + 2(n+1/2))^2} - \frac{1}{(2(s-1/2) - 2(n+1/2))^2} \right)$$
$$= - \frac{1}{\alpha} \sum_{n \geq 0} \left( \frac{1}{(2s + 2n)^2} - \frac{1}{(2s - 2(n+1))^2} \right) = - \frac{1}{4\alpha} \sum_{n \geq 0} \left( \frac{1}{(n+s)^2} - \frac{1}{(n+1)-s)^2} \right) $$
$$= - \frac{1}{4\alpha s^2} - \frac{1}{4\alpha} \left( \sum_{n \geq 1} \frac{1}{(n+s)^2} + \sum_{m \geq 1} \frac{1}{(m-s)^2} \right)$$
$$= - \frac{1}{4\alpha s^2} + \frac{1}{4\alpha} \sum_{n \geq 1} \frac{-1}{(n+s)^2} + \frac{1}{(n-s)^2},$$
and 
$$\frac{1}{4\alpha} \sum_{n \geq 1} \frac{-1}{(n+s)^2} + \frac{1}{(n-s)^2} = \frac{1}{4\alpha} \sum_{n \geq 1} \frac{4ns}{(n^2 - s^2)^2}.$$
Thus, the sum part simplifies to 
$$- \frac{1}{4\alpha s^2} + \frac{s}{\alpha} \sum_{n \geq 1} \frac{n}{(n^2 - s^2)^2}.$$

Hence, we have computed the Hilbert-Schmidt norm square, 
$$|| (\Delta_0 + s(s-1))^{-1} ||_{HS} ^2 $$
$$= (g-1) \left( \frac{\pi^2}{\alpha}(1+\tan^2(\pi \alpha)) 
-2 \left( - \frac{1}{4\alpha s^2} + \frac{s}{\alpha} \sum_{n \geq 1} \frac{n}{(n^2 - s^2)^2} \right) \right) - L^2 \log Z(s).$$
$$= \frac{(g-1)\pi^2}{\alpha} (1+\tan^2(\pi \alpha)) + \frac{(g-1)}{2
  \alpha s^2} - \frac{2(g-1) s}{\alpha} \sum_{n \geq 1} \frac{ n}{(n^2
  - s^2)^2} - L^2 \log Z(s).
$$
Consequently we find also the variational formulas
for the norm square
$$
\bar\partial_{\mu}\partial_{\mu}
|| (\Delta_0 + s(s-1))^{-1} ||_{HS} ^2 
=
 -\bar\partial_{\mu}\partial_{\mu}
( L^2
 \log Z(s))=
- L^2(
\bar\partial_{\mu}\partial_{\mu}
 \log Z(s)
)
$$
with 
$\bar\partial_{\mu}\partial_{\mu}
 \log Z(s)$ being given in Sections 4 and 5.
\end{proof} 
\end{appendix} 

\providecommand{\bysame}{\leavevmode\hbox to3em{\hrulefill}\thinspace}
\providecommand{\MR}{\relax\ifhmode\unskip\space\fi MR }
\providecommand{\MRhref}[2]{%
  \href{http://www.ams.org/mathscinet-getitem?mr=#1}{#2}
}
\providecommand{\href}[2]{#2}

\bibliographystyle{amsplain}

\end{document}